\documentclass[11pt]{article}

\usepackage{amsfonts,amssymb,amsthm,amsmath,bbm,ytableau}
\usepackage{color,graphicx}
\usepackage{hyperref}
\usepackage{setspace}
\usepackage{tikz-cd}

\theoremstyle{definition}
\newtheorem{definition}{Definition}[section]
\theoremstyle{plain}
\newtheorem{lemma}[definition]{Lemma}
\newtheorem{proposition}[definition]{Proposition}
\newtheorem{theorem}[definition]{Theorem}
\newtheorem{corollary}[definition]{Corollary}
\theoremstyle{remark}
\newtheorem{remark}[definition]{Remark}
\numberwithin{equation}{section}
\allowdisplaybreaks

\title{On the Mirabolic Trace Formula for $\mathfrak{gl}(n)$}
\author{Shuyang Cheng}
\date{}
\begin{document}
\maketitle

\begin{abstract}
    In this paper, a Chaudouard type trace formula is established for the Lie algebra $\mathfrak{gl}(n)$, by integrating the Lie algebra analogue of the Selberg kernel function against a mirabolic Eisenstein series on $\mathrm{GL}(n)$. The result is a combination of zeta functions $\zeta_E(s)$ of extensions over the base field $F$ of degree $[E:F]\leq n$.
\end{abstract}

\section*{A note to the reader}

The author would like to express his gratitude to an anonymous referee who pointed out that in Definition \ref{regularized mirabolic integral definition} where the regularized mirabolic integral $I(s)$ is defined, the summation $\sum_P$ over all parabolic subgroups $P\subset G$ is divergent. There are two possible approaches to address this issue:
\begin{itemize}
    \item The author believes it should be possible to treat the integrand of $I(s)$, in other words the divergent series 
    \begin{eqnarray*}
        \sum_P(-1)^{\ell(P)} K_P(g;f) \sum_{v^*\in(V^{U_P})'(F)} \Phi(v^*g)\big|\det(g)\big|^s,
    \end{eqnarray*}
    as a sum of distributions and then establish, via essentially the same proof for Proposition \ref{gl(n) singular integral proposition}, that the sum of distributions converges to a distribution which could be represented by an integrable function whose integral is equal to $\sum_{\lambda\vdash n} \sum_{\mathfrak{o}_\lambda} I_{\mathfrak{o}_\lambda}(s)$.
    \item Alternatively one could define $I(s)$ directly to be $\sum_{\lambda\vdash n} \sum_{\mathfrak{o}_\lambda} I_{\mathfrak{o}_\lambda}(s)$. Then the main burden of the proof becomes establishing the Fourier invariance $f\leftrightarrow\widehat{f}$ in Theorem \ref{mirabolic trace formula}, for which one could rearrange the ingredients of the proof for Proposition \ref{gl(n) singular integral proposition} and proceed by induction on $n$. This approach has the advantage of staying within the realm of functions in the classical sense.
\end{itemize}
The author is working on the second approach. This draft will be updated when the issue has been fixed. Meanwhile, the author welcomes any ideas and suggestions on possible improvements via email.

\tableofcontents

\section{Introduction}
\label{intro}

Let $G$ be a Lie group and $\Gamma\subset G$ a discrete subgroup such that the quotient space $\Gamma\backslash G$ has finite volume. The Selberg trace formula is an identity of the form
\begin{eqnarray}
    \label{selberg trace formula}
    \sum_oI_o
    &=&
    \sum_\pi I_\pi
\end{eqnarray}
where $o$ ranges over all conjugacy classes in $\Gamma$, $\pi$ ranges over all irreducible unitary representations of $G$ in $L^2(\Gamma\backslash G)$, and $I_o,I_\pi$ are distributions on $G$.

The proof of Selberg \cite{S56} starts with the Selberg kernel function
\begin{eqnarray}
    \label{selberg kernel function}
    K(g_1,g_2;f)
    &=&
    \sum_{\gamma\in\Gamma} f(g_1^{-1}\gamma g_2)
\end{eqnarray}
where $g_1,g_2\in\Gamma\backslash G$ and $f\in C_c^\infty(G)$ is a test function, and then derives the Selberg trace formula (\ref{selberg trace formula}) by expanding the integral
\begin{eqnarray}
    \label{selberg integral}
    \int_{\Gamma\backslash G} K(g,g;f)\mathrm{d}g
    &=&
    \mathrm{Tr}\big(R_f,L^2(\Gamma\backslash G)\big)
\end{eqnarray}
in two different ways. In many interesting examples the quotient space $\Gamma\backslash G$ is non-compact. In this case the integral (\ref{selberg integral}) will diverge and one integrates a truncated kernel function $K^T(g,g;f)$ over $\Gamma\backslash G$ instead.

Alternatively, Zagier \cite{Z81} has introduced another method to regularize the divergent integral (\ref{selberg integral}), in the first interesting example of $G=\mathrm{SL}_2(\mathbb{R})$ and $\Gamma=\mathrm{SL}_2(\mathbb{Z})$, namely by integrating
\begin{eqnarray}
    \label{zagier integral}
    \int_{\Gamma\backslash G} K(g,g;f)E(g,s) \mathrm{d}g
\end{eqnarray}
where $s\in\mathbb{C}$ and $E(g,s)$ is an Eisenstein series on $\Gamma\backslash G$ such that the integral (\ref{zagier integral}) converges on the half-plane $\mathrm{Re}(s)>1$ and continuous to a meromorphic function defined on the entire complex plane. Formally, the integral (\ref{zagier integral}) will converge to (\ref{selberg integral}) as $s\to1$, and the divergent nature of (\ref{selberg integral}) is reflected by the existence of a pole of (\ref{zagier integral}) at $s=1$. Then the divergent integral (\ref{selberg integral}) could be regularized by evaluating the residue of (\ref{zagier integral}) at $s=1$.

The method of Zagier \cite{Z81} is reminiscent of the method of zeta function regularization. Indeed, one of the most beautiful features of Zagier's method is that, after expanding (\ref{zagier integral}) into the identity
\begin{eqnarray}
    \label{zagier trace formula}
    \sum_oI_o(s)
    &=&
    \sum_\pi I_\pi(s)
\end{eqnarray}
which is analogous to (\ref{selberg trace formula}), the geometric distributions $I_o(s)$ are Dedekind zeta functions of quadratic fields, and the spectral distributions $I_\pi(s)$ are Rankin--Selberg $L$-functions of modular forms, or other special functions of a similar nature. The identity (\ref{zagier trace formula}) is the first example of a mirabolic trace formula.

In many introductory expositions, the Selberg trace formula is presented as a non-abelian generalization of the Poisson summation formula
\begin{eqnarray}
    \label{poisson summation formula}
    \sum_{n\in\mathbb{Z}}f(n)
    &=&
    \sum_{n\in\mathbb{Z}}\widehat{f}(n)
\end{eqnarray}
for the pair $G=\mathbb{R}$ and $\Gamma=\mathbb{Z}$. However, there is another less well-known but equally interesting relation between (\ref{poisson summation formula}) and (\ref{selberg trace formula}), namely the trace formula for Lie algebras. Let $\mathfrak{g}$ be the Lie algebra of $G$ and $\Lambda\subset\mathfrak{g}$ a lattice which is preserved under the adjoint action of the discrete subgroup $\Gamma\subset G$. Then analogous to (\ref{selberg kernel function}), one may define the Lie algebra kernel function
\begin{eqnarray}
    K(g;f)
    &=&
    \sum_{\nu\in\Lambda} f\big(\nu\cdot\mathrm{ad}(g)\big)
\end{eqnarray}
where $g\in\Gamma\backslash G$ and $f\in\mathcal{S}(\mathfrak{g})$ is a Schwartz function, and formally arrive at a possibly divergent identity of the form
\begin{eqnarray}
    \label{chaudouard trace formula}
    \sum_\mathfrak{o}I_\mathfrak{o}
    &=&
    \sum_\mathfrak{o}\widehat{I_\mathfrak{o}}
\end{eqnarray}
where $\mathfrak{o}$ ranges over all orbits in $\Lambda$ under the adjoint action of $\Gamma$, and $I_\mathfrak{o},\widehat{I_\mathfrak{o}}$ are tempered distributions on $\mathfrak{g}$ such that
\begin{eqnarray}
    \langle\widehat{I_\mathfrak{o}},f\rangle
    &=&
    \langle I_\mathfrak{o},\widehat{f}\rangle.
\end{eqnarray}
A regularized version of the trace formula for Lie algebras (\ref{chaudouard trace formula}) has been established by Chaudouard \cite{C02} following the truncation method of Selberg which has been generalized by Arthur \cite{A78}, working over the adeles.

The goal of this paper is to regularize the identity (\ref{chaudouard trace formula}), following Zagier's method which has been reformulated by Jacquet--Zagier \cite{JZ87} over $\mathrm{GL}_2(\mathbb{A})$, for the adelic pairs $G=\mathrm{GL}_n(\mathbb{A}),\Gamma=\mathrm{GL}_n(F)$ and $\mathfrak{g}=\mathfrak{gl}_n(\mathbb{A}),\Lambda=\mathfrak{gl}_n(F)$. The result is a mirabolic trace formula for $\mathfrak{gl}(n)$
\begin{equation}
    \sum_\mathfrak{o}I_\mathfrak{o}(s)
    =
    \sum_\mathfrak{o}\widehat{I_\mathfrak{o}}(s).
    \tag{Theorem \ref{mirabolic trace formula}}
\end{equation}
Following the main references \cite{C02} and \cite{JZ87}, the rest of this paper will be written in the language of adelic algebraic groups.

\section{Preliminaries and a motivating example}
\label{sec:2}

\paragraph{Algebraic preliminaries} Let $G$ denote the algebraic group $\mathrm{GL}(n)$ and $\mathfrak{g}$ denote the Lie algebra $\mathfrak{gl}(n)$. More generally, the Lie algebra of an algebraic group will be denoted by the same letter in Fraktur font. Denote the right adjoint action by
\begin{eqnarray*}
    X\cdot\mathrm{ad}(g)
    &=&
    g^{-1}Xg
\end{eqnarray*}
where $X\in\mathfrak{g}$ and $g\in G$. Let $V$ denote the vector space consisting of $n$-dimensional row vectors, then $G$ operates on $V$ from the right by right multiplication.

Let $Z\subset G$ denote the center consisting of scalar matrices, $T\subset G$ denote the maximal torus consisting of diagonal matrices, $B^\pm\subset G$ denote the Borel subgroups consisting of upper or lower triangular matrices, and $U^\pm\subset B^\pm$ denote the unipotent radicals consisting of matrices with ones along the diagonal.

A standard parabolic subgroup is a subgroup which contains $B^+$. More generally, a parabolic subgroup is a subgroup which is conjugate to some standard parabolic subgroup. If $P\subset G$ is parabolic, let $U_P\subset P$ denote the unipotent radical and $M_P=P/U_P$ denote the Levi component. If $P\subset G$ is standard parabolic, then $P$ consists of block upper triangular matrices, and $M_P\subset P$ will be identified with the subgroup consisting of block diagonal matrices. Each parabolic subgroup $P\subset G$ is associated with a partial flag
\begin{eqnarray*}
    V=V_r\supset V_{r-1}\supset\dots\supset V_1\supset V_0=0
\end{eqnarray*}
with successive subquotients $W_i=V_i/V_{i-1}$ such that $M_P\simeq\prod_{i=1}^r\mathrm{GL}(W_i)$. Let $\ell(P)=r+1$ denote the length of the partial flag associated to $P$.

A partition of $n$ consists of positive integers $n_1\geq\dots\geq n_r$ such that $n_1+\dots+n_r=n$. Due to inconsistent conventions between the theory of partitions and linear algebra, partitions will be written in reverse order and denoted by $[n_r,\dots,n_1]=\lambda\vdash n$. If $\lambda\vdash n$ is a partition, let $P_\lambda\subset G$ denote the standard parabolic subgroup consisting of block upper triangular matrices such that the $(i,j)$th block consists of $n_i\times n_j$ matrices.

An element $X\in\mathfrak{g}$ is regular if its centralizer $G_X\subset G$ is $n$-dimensional, or equivalently if there exists $v^*\in V$ such that $v^*,v^*X,\dots,v^*X^{n-1}$ form a basis of $V$. Each regular element is conjugate to a unique companion matrix of the form
\begin{eqnarray*}
    X_p&=&
    \left[\begin{matrix}
    a_1&\cdots&a_{n-1}&a_n\\
    1&\cdots&0&0\\
    \vdots&\ddots&\vdots&\vdots\\
    0&\cdots&1&0
    \end{matrix}\right]
\end{eqnarray*}
where $p(t)=t^n-a_1t^{n-1}-\dots-a_n$ is the characteristic polynomial of $X$. More generally by the theory of Frobenius normal form, each $X\in\mathfrak{g}$ is conjugate to a unique block diagonal matrix of the form
\begin{eqnarray*}
    X_{p_r,\dots,p_1}
    &=&
    \left[\begin{matrix}
    X_{p_r}&0&\cdots&0\\
    0&X_{p_{r-1}}&\cdots&0\\
    \vdots&\vdots&\ddots&\vdots\\
    0&0&\cdots&X_{p_1}
    \end{matrix}\right]
\end{eqnarray*}
where $p_r\mid p_{r-1}\mid\dots\mid p_1$. Let $n_i=\deg(p_i)$, then $X$ has Frobenius normal form of partition type $\lambda=[n_r,\dots,n_1]\vdash n$, and the polynomials $p_r,\dots,p_1$ are the invariant factors of $X$.

\paragraph{Analytic preliminaries}
Let $F$ be a global field and let $\mathbb{A}=\prod_v'F_v$ denote the ring of adeles of $F$, where the product ranges over all places $v$ of $F$ which is restricted with respect to the local rings of integers $O_v\subset F_v$ for all $v<\infty$. Let $|\cdot|_v:F_v\to\mathbb{R}_{\geq0}$ and $|\cdot|:\mathbb{A}\to\mathbb{R}_{\geq0}$ denote the local and global norms such that $|\cdot|=\prod_v|\cdot|_v$ and the product formula
\begin{eqnarray*}
    \prod_v|x|_v
    &=&
    1
\end{eqnarray*}
holds for all $x\in F^\times$. Fix a non-trivial additive character $\psi:\mathbb{A}\to\mathbb{C}^\times$ such that
\begin{eqnarray*}
    \psi(x)
    &=&
    1
\end{eqnarray*}
holds for all $x\in F$.

Let $K=\prod_{v<\infty}G(O_v)\times K_\infty\subset G(\mathbb{A})$ be a maximal compact subgroup such that the Iwasawa decomposition
\begin{eqnarray*}
    G(\mathbb{A})
    &=&
    P(\mathbb{A})K
\end{eqnarray*}
holds for all parabolic subgroups $P\subset G$. If $P\subset G$ is standard parabolic, fix compatible Haar measures such that
\begin{eqnarray*}
    \int_{G(\mathbb{A})}h(g)\mathrm{d}g
    &=&
    \int_K\int_{U_P(\mathbb{A})}\int_{M_P(\mathbb{A})} h(muk)\mathrm{d}m\mathrm{d}u\mathrm{d}k\\
    &=&
    \int_K\int_{M_P(\mathbb{A})}\int_{U_P(\mathbb{A})} h(umk) \prod_{i<j}\frac{|\det(m_i)|^{n_j}}{|\det(m_j)|^{n_i}}\mathrm{d}u\mathrm{d}m\mathrm{d}k
\end{eqnarray*}
for all $h\in L^1(G(\mathbb{A}))$, where $m=\left[\begin{smallmatrix}
m_r&\dots&0\\
\dots&\dots&\dots\\
0&\dots&m_1
\end{smallmatrix}\right]\in M_P(\mathbb{A})$ is block diagonal with $m_i\in\mathrm{GL}_{n_i}(\mathbb{A})$.

If $W$ is a finite-dimensional vector space, let $\mathcal{S}(W(\mathbb{A}))=\bigotimes_v'\mathcal{S}(W(F_v))$ denote the vector space of Schwartz functions, where the tensor product is restricted with respect to the characteristic functions $\mathbbm{1}_{W(O_v)}$ for all $v<\infty$. Let $\langle\cdot,\cdot\rangle$ be a non-degenerate bilinear form on $W$, then $W(\mathbb{A})$ admits a unique Haar measure $\mathrm{d}w$ which is self dual with respect to $\langle\cdot,\cdot\rangle$ and $\psi$ such that the Fourier transform $\widehat{f}\in\mathcal{S}(W(\mathbb{A}))$ defined by
\begin{eqnarray*}
    \widehat{f}(v)
    &=&
    \int_{W(\mathbb{A})} f(w)\overline{\psi}(\langle v,w\rangle)\mathrm{d}w
\end{eqnarray*}
satisfies the Fourier inversion formula
\begin{eqnarray*}
    \widehat{\widehat{f}(}w)
    &=&
    f(-w)
\end{eqnarray*}
and the Poisson summation formula
\begin{eqnarray*}
    \sum_{w\in W(F)} f(w)
    &=&
    \sum_{w\in W(F)} \widehat{f}(w)
\end{eqnarray*}
for all $f\in\mathcal{S}(W(\mathbb{A}))$.

Equip the vector space $\mathfrak{g}$ with the non-degenerate bilinear form
\begin{eqnarray*}
    \langle X,Y\rangle
    &=&
    \mathrm{tr}(XY)
\end{eqnarray*}
which is invariant under the adjoint action and the associated self-dual Haar measure $\mathrm{d}X$ on $\mathfrak{g}(\mathbb{A})$. Then the induced Fourier transform operator on $\mathcal{S}(\mathfrak{g}(\mathbb{A}))$ intertwines with the representation of $G(\mathbb{A})$ on $\mathcal{S}(\mathfrak{g}(\mathbb{A}))$ via the adjoint action
\begin{eqnarray*}
    \begin{tikzcd}
    \mathcal{S}(\mathfrak{g}(\mathbb{A})) \arrow[r,"\widehat{f}"] \arrow[d,"f(X\cdot\mathrm{ad}(g))"'] & \mathcal{S}(\mathfrak{g}(\mathbb{A})) \arrow[d,"\widehat{f}(X\cdot\mathrm{ad}(g))"]\\
    \mathcal{S}(\mathfrak{g}(\mathbb{A})) \arrow[r,"\widehat{f}"] & \mathcal{S}(\mathfrak{g}(\mathbb{A})).
    \end{tikzcd}
\end{eqnarray*}
If $P\subset G$ is a parabolic subgroup and $f\in\mathcal{S}(\mathfrak{g}(\mathbb{A))}$, define the parabolic descent $f_P\in\mathcal{S}(\mathfrak{m}_P(\mathbb{A}))$ by
\begin{eqnarray*}
    f_P(X)
    &=&
    \int_K\int_{\mathfrak{u}_P(\mathbb{A})} f\big((\tilde{X}+U)\cdot\mathrm{ad}(k)\big) \mathrm{d}U\mathrm{d}k
\end{eqnarray*}
where $\tilde{X}\in\mathfrak{p}(\mathbb{A})$ is a lift of $X\in\mathfrak{m}_P(\mathbb{A})$. Then the parabolic descent operator intertwines with the Fourier transform operators on $\mathcal{S}(\mathfrak{g}(\mathbb{A}))$ and $\mathcal{S}(\mathfrak{m}_P(\mathbb{A}))$
\begin{eqnarray*}
    \begin{tikzcd}
    \mathcal{S}(\mathfrak{g}(\mathbb{A})) \arrow[r,"\widehat{f}"] \arrow[d,"f_P"'] & \mathcal{S}(\mathfrak{g}(\mathbb{A})) \arrow[d,"\widehat{f}_P"]\\
    \mathcal{S}(\mathfrak{m}_P(\mathbb{A})) \arrow[r,"\widehat{f}"] & \mathcal{S}(\mathfrak{m}_P(\mathbb{A})).
    \end{tikzcd}
\end{eqnarray*}

Equip the vector space $V$ with the non-degenerate bilinear form
\begin{eqnarray*}
    \langle u^*,v^*\rangle
    &=&
    u^*(^\intercal v^*)
\end{eqnarray*}
where $\intercal$ denotes matrix transpose and the associated self-dual Haar measure $\mathrm{d}v^*$ on $V(\mathbb{A})$. Then the induced Fourier transform operator on $\mathcal{S}(V(\mathbb{A}))$ intertwines with the standard representation of $G(\mathbb{A})$ on $\mathcal{S}(V(\mathbb{A}))$ via right multiplication and its contragredient twisted by $|\det|^{-1}$
\begin{eqnarray*}
    \begin{tikzcd}
    \mathcal{S}(V(\mathbb{A})) \arrow[r,"\widehat{f}"] \arrow[d,"f(v^*g)"'] & \mathcal{S}(V(\mathbb{A})) \arrow[d,"|\det(g)|^{-1}\widehat{f}(v^*{}^\intercal g^{-1})"]\\
    \mathcal{S}(V(\mathbb{A})) \arrow[r,"\widehat{f}"] & \mathcal{S}(V(\mathbb{A})).
    \end{tikzcd}
\end{eqnarray*}

\paragraph{The trace formula of Chaudouard}
Let $f\in\mathcal{S}(\mathfrak{g}(\mathbb{A}))$, define the kernel function
\begin{eqnarray*}
    K(g;f)
    &=&
    \sum_{X\in\mathfrak{g}(F)} f\big(X\cdot\mathrm{ad}(g)\big)
\end{eqnarray*}
and the parabolic kernel functions
\begin{eqnarray*}
    K_P(g;f)
    &=&
    \sum_{X\in\mathfrak{p}(F)} \int_{\mathfrak{u}_P(\mathbb{A})/\mathfrak{u}_P(F)} f\big((X+U)\cdot\mathrm{ad}(g)\big) \mathrm{d}U
\end{eqnarray*}
for $g\in G(F)\backslash G(\mathbb{A})$, then
\begin{eqnarray*}
    K(g;f)=K(g;\widehat{f})
    &\mathrm{and}&
    K_P(g;f)=K_P(g;\widehat{f})
\end{eqnarray*}
by the Poisson summation formula.

Following the truncation procedure of Arthur \cite{A78}, Chaudouard \cite{C02} has regularized the divergent integral
\begin{eqnarray*}
    \int_{G(F)\backslash G(\mathbb{A})^1} \sum_P(-1)^{\ell(P)} K_P(g;f) \mathrm{d}g,
\end{eqnarray*}
where $G(\mathbb{A})^1\subset G(\mathbb{A})$ denotes the subgroup consisting of $g\in G(\mathbb{A})$ such that $|\det(g)|=1$, into an absolutely convergent identity
\begin{eqnarray*}
    \sum_{\tilde{\mathfrak{o}}} J_{\tilde{\mathfrak{o}}}(f)
    &=&
    \sum_{\tilde{\mathfrak{o}}} J_{\tilde{\mathfrak{o}}}(\widehat{f})
\end{eqnarray*}
where the summation $\sum_{\tilde{\mathfrak{o}}}$ ranges over all equivalence classes $\tilde{\mathfrak{o}}\subset\mathfrak{g}(F)$ where each $\tilde{\mathfrak{o}}$ consists of all $X\in\mathfrak{g}(F)$ with a common characteristic polynomial $p_{\tilde{\mathfrak{o}}}(t)\in F[t]$.
\begin{itemize}
    \item If the characteristic polynomial $p_{\tilde{\mathfrak{o}}}(t)$ is irreducible over $F$, then $\tilde{\mathfrak{o}}=\mathfrak{o}_\mathrm{ell}$ is an elliptic conjugacy class in $\mathfrak{g}(F)$ with centralizer isomorphic to the Weil restriction $\mathrm{Res}_{E/F}\mathbb{G}_m$ where $E\simeq F[t]/(p_{\mathfrak{o}_\mathrm{ell}})$, which is an elliptic torus modulo $Z$ with $\mathrm{Res}_{E/F}\mathbb{G}_m(F)\simeq E^\times$.
    
    In this case the truncation procedure trivializes and
    \begin{eqnarray*}
        J_{\mathfrak{o}_\mathrm{ell}}(f)
        &=&
        \int_{G(F)\backslash G(\mathbb{A})^1} \sum_{X\in\mathfrak{o}_\mathrm{ell}} f\big(X\cdot\mathrm{ad}(g)\big) \mathrm{d}g\\
        &=&
        \mathrm{vol}(\mathbb{A}_E^1/E^\times) \int_{G_X(\mathbb{A})\backslash G(\mathbb{A})} f\big(X\cdot\mathrm{ad}(g)\big) \mathrm{d}g
    \end{eqnarray*}
    where $\mathbb{A}_E^1\subset\mathbb{A}_E^\times$ denotes the group of ideles of $E$ of norm 1.
    \item More generally if the characteristic polynomial $p_{\tilde{\mathfrak{o}}}(t)$ has distinct roots in its splitting field, or equivalently if its discriminant $\Delta(p_{\tilde{\mathfrak{o}}})\neq0$, then $\tilde{\mathfrak{o}}=\mathfrak{o}_\mathrm{rs}$ is a regular semisimple conjugacy class in $\mathfrak{g}(F)$ with centralizer isomorphic to a possibly non-elliptic torus.
    
    In this case Chaudouard \cite{C02} has obtained a similar expression
    \begin{eqnarray*}
        J_{\mathfrak{o}_\mathrm{rs}}(f)
        &=&
        \mathrm{vol}(Z_{M_1}^\infty G_{X_1}(F)\backslash G_{X_1}(\mathbb{A})) \int_{G_{X_1}(\mathbb{A})\backslash G(\mathbb{A})} f\big(X_1\cdot\mathrm{ad}(g)\big) v(g,T_0)\mathrm{d}g
    \end{eqnarray*}
    as an orbital integral weighted by the weight factor $v(g,T_0)$ introduced by Arthur \cite{A78}.
    \item Otherwise $\Delta(p_{\tilde{\mathfrak{o}}})=0$ and $\tilde{\mathfrak{o}}=\bigcup_{[\lambda_i]}\mathfrak{o}_{[\lambda_i]}$ is a finite union of conjugacy classes indexed by partitions $\lambda_i$ which correspond to different Jordan block sizes for the $i$th distinct eigenvalue.
    
    In this case the distributions $J_{\tilde{\mathfrak{o}}}(f)$ are more difficult to understand. However, for most applications one could impose finitely many local conditions on the test functions $f=\otimes_v f_v$ and $\widehat{f}=\otimes_v\widehat{f_v}$ to remove all non-regular semisimple or non-elliptic summands and work with the simple trace formula
    \begin{eqnarray*}
        \sum_{\mathfrak{o}_\mathrm{rs}} J_{\mathfrak{o}_\mathrm{rs}}(f)
        &=&
        \sum_{\mathfrak{o}_\mathrm{rs}} J_{\mathfrak{o}_\mathrm{rs}}(\widehat{f})
    \end{eqnarray*}
    or the very simple trace formula
    \begin{eqnarray*}
        \sum_{\mathfrak{o}_\mathrm{ell}} J_{\mathfrak{o}_\mathrm{ell}}(f)
        &=&
        \sum_{\mathfrak{o}_\mathrm{ell}} J_{\mathfrak{o}_\mathrm{ell}}(\widehat{f}).
    \end{eqnarray*}
\end{itemize}

\paragraph{Tate integrals for zeta functions}
Let $\Phi=\otimes_v \Phi_v\in\mathcal{S}(\mathbb{A})$, then the Tate integral $I(s;\Phi)$ is an adelic integral of the form
\begin{eqnarray*}
    I(s;\Phi)
    &=&
    \int_{\mathbb{A}^\times} \Phi(x) |x|^s\mathrm{d}^\times x\\
    &=&
    \prod_v\int_{F_v^\times} \Phi_v(x_v) |x_v|_v^s\mathrm{d}^\times x_v
\end{eqnarray*}
for all $s\in\mathbb{C}$ such that the integral converges absolutely. In his thesis \cite{T50}, Tate has established the following analytic properties of $I(s;\Phi)$:
\begin{itemize}
    \item the integral $I(s;\Phi)$ converges absolutely on the half-plane $\mathrm{Re}(s)>1$ and continuous to a meromorphic function on the entire complex plane, which will also be denoted by $I(s;\Phi)$;
    \item there exists a function $\zeta_F(s)$ which is holomorphic on the half-plane $\mathrm{Re}(s)>1$ such that
    \begin{eqnarray*}
        I(s;\Phi)
        &=&
        \zeta_F(s)\prod_{v\in S}e_v(s;\Phi_v)
    \end{eqnarray*}
    for all $\Phi\in\mathcal{S}(\mathbb{A})$, where $S$ is a finite set of places of $F$ and each $e_v(s;\Phi_v)$ is an elementary function in $s$ which continuous holomorphically to the entire complex plane;
    \item the meromorphic function $I(s;\Phi)$ satisfies the functional equation
    \begin{eqnarray*}
        I(s;\Phi)
        &=&
        I(1-s;\widehat{\Phi})
    \end{eqnarray*}
    which is another consequence of the Poisson summation formula.
\end{itemize}
From the analytic properties of $I(s;\Phi)$, Tate \cite{T50} then deduces the classical Euler factorization and functional equation of the zeta function $\zeta_F(s)$.

The group $G$ operates on $V$ with a unique Zariski open orbit denoted by $V'=V-\{0\}$. Let $\Phi\in\mathcal{S}(V(\mathbb{A}))$, then the mirabolic Eisenstein series $E(g,s;\Phi)$ is an automorphic function in $g\in Z(\mathbb{A})G(F)\backslash G(\mathbb{A})$ defined by
\begin{eqnarray*}
    E(g,s;\Phi)
    &=&
    |\det(g)|^s \int_{\mathbb{A}^\times/F^\times} \sum_{v^*\in V'(F)} \Phi(zv^*g)|z|^{ns} \mathrm{d}^\times z
\end{eqnarray*}
which converges absolutely for $\mathrm{Re}(s)>1$ and continuous meromorphically to the entire complex plane. As another consequence of the Poisson summation formula, the mirabolic Eisenstein series $E(g,s;\Phi)$ satisfies the functional equation
\begin{eqnarray*}
    E(g,s;\Phi)
    &=&
    E(^\intercal g^{-1},1-s;\widehat{\Phi}).
\end{eqnarray*}
If $X\in\mathfrak{o}_\mathrm{ell}$ is an elliptic element in $\mathfrak{g}(F)$ with characteristic polynomial $p_{\mathfrak{o}_\mathrm{ell}}\in F[t]$, then the integral
\begin{eqnarray*}
    I_X(s;\Phi)
    &=&
    \int_{G_X(F)\backslash G_X(\mathbb{A})^1} E(g,s;\Phi) \mathrm{d}g
\end{eqnarray*}
is a Tate integral for $\zeta_E(s)$ where $E\simeq F[t]/(p_{\mathfrak{o}_\mathrm{ell}})$.

More generally, Tate \cite{T50} has introduced Tate integrals twisted by a multiplicative character $\chi:\mathbb{A}^1/F^\times\to\mathbb{C}^\times$ where $\mathbb{A}^1\subset\mathbb{A}^\times$ denotes the group of ideles of norm 1, and similarly there exist mirabolic Eisenstein series twisted by the character $\chi\circ\det:G(\mathbb{A})\to\mathbb{C}^\times$. The results of this paper could be extended to such cases, which would lead to a mirabolic version of the twisted trace formula for $\mathfrak{gl}(n)$. However, such an extension will be excluded from this paper for simplicity.

\paragraph{The example of $\mathfrak{gl}(2)$}
For the rest of this section let $\mathfrak{g}$ denote $\mathfrak{gl}(2)$, $G$ denote $\mathrm{GL}(2)$, and choose Schwartz functions $f\in\mathcal{S}(\mathfrak{g}(\mathbb{A}))$ and $\Phi\in\mathcal{S}(\mathbb{A}^2)$. The following have been computed for $\mathrm{Re}(s)>1$ in \cite{JZ87}:

\begin{itemize}
    \item If $\mathfrak{o}_\mathrm{ell}\subset\mathfrak{g}(F)$ is an elliptic conjugacy class containing $d+\sqrt{\delta}=\left[\begin{smallmatrix}d&\delta\\1&d\end{smallmatrix}\right]$ where $\delta$ is non-square in $F$, then
    \begin{eqnarray*}
        I_{\mathfrak{o}_\mathrm{ell}}(s)
        &=&
        \int_{Z(\mathbb{A})G(F)\backslash G(\mathbb{A})} \sum_{X\in\mathfrak{o}_\mathrm{ell}}f\big(X\cdot\mathrm{ad}(g)\big) E(g,s;\Phi)\mathrm{d}g\\
        &=&
        \int_{G(F)\backslash G(\mathbb{A})}\sum_{\begin{subarray}{c}X\in\mathfrak{o}_\mathrm{ell}\\v^*\in F^2{}'\end{subarray}} f\big(X\cdot\mathrm{ad}(g)\big) \Phi(v^*g) \big|\det(g)\big|^s \mathrm{d}g\\
        &=&
        \int_{\mathbb{A}[\sqrt{\delta}]^\times\backslash G(\mathbb{A})} f\big(\left[\begin{smallmatrix}d&\delta\\1&d\end{smallmatrix}\right]\cdot\mathrm{ad}(g)\big) \Big(\int_{\mathbb{A}[\sqrt{\delta}]^\times}\Phi(\alpha g)\big|\det(\alpha g)\big|^s\mathrm{d}^\times\alpha\Big)\mathrm{d}g
    \end{eqnarray*}
    is an elliptic orbital integral weighted by the Tate integral
    \begin{eqnarray*}
        \int_{\mathbb{A}[\sqrt{\delta}]^\times}\Phi(\alpha g)\big|\det(\alpha g)\big|^s\mathrm{d}^\times\alpha
        &=&
        \zeta_{F(\sqrt{\delta})}(s)e(g,s;\Phi)
    \end{eqnarray*}
    where $e(g,s;\Phi)$ is a finite Euler product of elementary functions in $s$.
    \item If $\mathfrak{o}_\mathrm{hyp}\subset\mathfrak{g}(F)$ is a hyperbolic conjugacy class containing $\left[\begin{smallmatrix}d_1&0\\0&d_2\end{smallmatrix}\right]$ where $d_1\neq d_2$, then
    \begin{eqnarray*}
        I_{\mathfrak{o}_\mathrm{hyp}}(s)
        &=&
        \int_{Z(\mathbb{A})G(F)\backslash G(\mathbb{A})} \sum_{X\in\mathfrak{o}_\mathrm{hyp}}\Big(f\big(X\cdot\mathrm{ad}(g)\big) E(g,s;\Phi)+\\
        &&\quad-\begin{array}{c}\textrm{singular~terms~arising~from}\\\textrm{all~Borel~}\mathfrak{b}\subset\mathfrak{g}\textrm{~containing~}X\end{array}\Big)\mathrm{d}g\\
        &=&
        \int_{G(F)\backslash G(\mathbb{A})}\sideset{}{'}\sum_{\begin{subarray}{c}X\in\mathfrak{o}_\mathrm{hyp}\\v^*\in F^2{}'\end{subarray}} f\big(X\cdot\mathrm{ad}(g)\big) \Phi(v^*g) \big|\det(g)\big|^s \mathrm{d}g
    \end{eqnarray*}
    where the primed summation ranges over all pairs $(X,v^*)$ such that $v^*$ is not an eigenvector of $X$, is equal to
    \begin{eqnarray*}
        &=&
        \int_{(\mathbb{A}^\times\times\mathbb{A}^\times)\backslash G(\mathbb{A})} f\big(\left[\begin{smallmatrix}d_1&0\\0&d_2\end{smallmatrix}\right]\cdot\mathrm{ad}(g)\big) \Big(\int_{\mathbb{A}^\times}\int_{\mathbb{A}^\times} \Phi\big([a,b]g\big)|a|^s|b|^s\mathrm{d}^\times a\mathrm{d}^\times b\Big)\mathrm{d}g
    \end{eqnarray*}
    which is a hyperbolic orbital integral weighted by the Tate integral
    \begin{eqnarray*}
        \int_{\mathbb{A}^\times}\int_{\mathbb{A}^\times} \Phi\big([a,b]g\big)|a|^s|b|^s\mathrm{d}^\times a\mathrm{d}^\times b
        &=&
        \zeta_F(s)^2e(g,s;\Phi)
    \end{eqnarray*}
    where $e(g,s;\Phi)$ is a finite Euler product of elementary functions in $s$.
    \item If $\mathfrak{o}_\mathrm{par}\subset\mathfrak{g}(F)$ is a parabolic conjugacy class containing $\left[\begin{smallmatrix}d&1\\0&d\end{smallmatrix}\right]$, then
    \begin{eqnarray*}
        I_{\mathfrak{o}_\mathrm{par}}(s)
        &=&
        \int_{Z(\mathbb{A})G(F)\backslash G(\mathbb{A})} \sum_{X\in\mathfrak{o}_\mathrm{par}}\Big(f\big(X\cdot\mathrm{ad}(g)\big) E(g,s;\Phi)+\\
        &&\quad-\begin{array}{c}\textrm{singular~terms~arising~from}\\\textrm{all~Borel~}\mathfrak{b}\subset\mathfrak{g}\textrm{~containing~}X\end{array}\Big)\mathrm{d}g\\
        &=&
        \int_{G(F)\backslash G(\mathbb{A})}\sideset{}{'}\sum_{\begin{subarray}{c}X\in\mathfrak{o}_\mathrm{par}\\v^*\in F^2{}'\end{subarray}} f\big(X\cdot\mathrm{ad}(g)\big) \Phi(v^*g) \big|\det(g)\big|^s \mathrm{d}g
    \end{eqnarray*}
    where the primed summation ranges over all pairs $(X,v^*)$ such that $v^*$ is not an eigenvector of $X$, is equal to
    \begin{eqnarray*}
        &=&
        \int_{(\mathbb{A}^\times\times\mathbb{A})\backslash G(\mathbb{A})} f\big(\left[\begin{smallmatrix}d&1\\0&d\end{smallmatrix}\right]\cdot\mathrm{ad}(g)\big)\times \Big(\int_{\mathbb{A}^\times}\int_\mathbb{A}\Phi\big([a,b]g\big)\frac{\mathrm{d}b}{|a|}\big|a^2\big|^s\mathrm{d}^\times a\Big)\mathrm{d}g
    \end{eqnarray*}
    which is a parabolic orbital integral weighted by the Tate integral
    \begin{eqnarray*}
        \int_{\mathbb{A}^\times}\Big(\int_\mathbb{A}\Phi\big([a,b]g\big)\mathrm{d}b\Big)|a|^{2s-1}\mathrm{d}^\times a
        &=&
        \zeta_F(2s-1)e(g,2s-1;\tilde{\Phi})
    \end{eqnarray*}
    where $\tilde{\Phi}\in\mathcal{S}(\mathbb{A})$ is a partial integral of $\Phi$ along parallel lines $\mathbb{A}\subset\mathbb{A}^2$. The orbital integral $I_{\mathfrak{o}_\mathrm{par}}(s)$ could be further evaluated as
    \begin{eqnarray*}
        &&
        \zeta_F(2s-1)\int_K\Big(\int_{\mathbb{A}^\times} f\big(\left[\begin{smallmatrix}d&c\\0&d\end{smallmatrix}\right]\cdot\mathrm{ad}(k)\big)|c|^s~\mathrm{d}^\times c\Big) e(g,2s-1;\tilde{\Phi})\mathrm{d}k\\
        &=&
        \zeta_F(s)\zeta_F(2s-1)\int_K e(d,k,s;f) e(g,2s-1;\tilde{\Phi})\mathrm{d}k
    \end{eqnarray*}
    where $e(d,k,s;f)$ and $e(g,2s-1;\tilde{\Phi})$ are both finite Euler products of elementary functions in $s$.
    \item If $\mathfrak{o}_z\subset\mathfrak{g}(F)$ is a central conjugacy class consisting of $\left[\begin{smallmatrix}d&0\\0&d\end{smallmatrix}\right]$, then
    \begin{eqnarray*}
        I_{\mathfrak{o}_z}(s)
        &=&
        \int_{G(F)\backslash G(\mathbb{A})} \sum_{v^*\in F^2{}'} \Phi(v^*g) \sum_{\begin{subarray}{c}X\in\mathfrak{b}_{v^*}(F)\\X\in\mathfrak{o}_z+\mathfrak{u}_{v^*}\end{subarray}} \Big(f\big(X\cdot\mathrm{ad}(g)\big)+\\
        &&\quad -\int_{\mathfrak{u}_{v^*}(\mathbb{A})/\mathfrak{u}_{v^*}(F)}f\big((X+U)\cdot\mathrm{ad}(g)\big)\mathrm{d}U\Big)\big|\det(g)\big|^s \mathrm{d}g
    \end{eqnarray*}
    where $\mathfrak{b}_{v^*}\subset\mathfrak{g}$ denotes the Borel subalgebra which stabilizes the line in $F^2$ spanned by $v^*$ and $\mathfrak{u}_{v^*}$ denotes the nilpotent radical of $\mathfrak{b}_{v^*}$, is equal to
    \begin{eqnarray*}
        &=&
        \int_K \Big(\int_{\mathbb{A}^\times}
        \widetilde{f^k}(d,c^*)\big|c^*\big|^s\mathrm{d}^\times c^*\Big) \Big(\int_{\mathbb{A}^\times} \Phi\big([0,a]k\big)\big|a\big|^{2s} \mathrm{d}^\times a\Big)\mathrm{d}k
        \\
        &=&
        \zeta_F(s)\zeta_F(2s)\int_K e(d,k,s;\widetilde{f^k})e(k,2s;\Phi)\mathrm{d}k
    \end{eqnarray*}
    where $f^k$ denotes the composite function $f\circ\mathrm{ad}(k)$ and $\widetilde{f}$ denotes the partial Fourier transform of $f$ defined by
    \begin{eqnarray*}
        \widetilde{f}(d,c^*)
        &=&
        \int_\mathbb{A} f\big(\left[\begin{smallmatrix}d&u\\0&d\end{smallmatrix}\right]\big) \overline{\psi}(c^*u) \mathrm{d}u.
    \end{eqnarray*}
\end{itemize}

\paragraph{The mirabolic trace formula for $\mathfrak{gl}(2)$}

Combining the previous results with the absolute convergence of the summation over all conjugacy classes which will be established for all $\mathfrak{gl}(n)$ later, it follows that the integral
\begin{eqnarray}
    \label{gl(2) mirabolic integral}
    &&
    \int_{G(F)\backslash G(\mathbb{A})} \sum_{v^*\in F^2{}'} \Big(K(g;f)-K_{B_{v^*}}(g;f)\Big)\Phi(v^*g)\big|\det(g)\big|^s \mathrm{d}g\nonumber\\
    &=&
    \sum_{\mathfrak{o}_\mathrm{ell}}I_{\mathfrak{o}_\mathrm{ell}}(s)+\sum_{\mathfrak{o}_\mathrm{hyp}}I_{\mathfrak{o}_\mathrm{hyp}}(s)+\sum_{\mathfrak{o}_\mathrm{par}}I_{\mathfrak{o}_\mathrm{par}}(s)+\sum_{\mathfrak{o}_z}I_{\mathfrak{o}_z}(s)
\end{eqnarray}
converges absolutely for $\mathrm{Re}(s)>1$ and continues meromorphically to the entire complex plane. Since $K(g;f)=K(g;\widehat{f})$ and $K_{B_{v^*}}(g;f)=K_{B_{v^*}}(g;\widehat{f})$ by the Poisson summation formula, it follows that the distribution defined by (\ref{gl(2) mirabolic integral}) is invariant under the Fourier transform $f\leftrightarrow\widehat{f}$ on $\mathcal{S}(\mathfrak{g}(\mathbb{A}))$.

\section{Contributions from the regular locus}
\label{sec:3}

\paragraph{The regular locus of the mirabolic action}
For the rest of this paper let $\mathfrak{g}$ denote $\mathfrak{gl}(n)$, $G$ denote $\mathrm{GL}(n)$, and choose Schwartz functions $f\in\mathcal{S}(\mathfrak{g}(\mathbb{A}))$ and $\Phi\in\mathcal{S}(V(\mathbb{A}))$. Motivated by the mirabolic integral
\begin{eqnarray*}
    &&
    \int_{Z(\mathbb{A})G(F)\backslash G(\mathbb{A})} K(g;f)E(g,s;\Phi)\mathrm{d}g\\
    &=&
    \int_{G(F)\backslash G(\mathbb{A})} \sum_{X\in\mathfrak{g}(F)} f\big(X\cdot\mathrm{ad}(g)\big) \sum_{v^*\in V'(F)}\Phi(v^*g)\big|\det(g)\big|^s \mathrm{d}g,
\end{eqnarray*}
consider the diagonal mirabolic action of $G$ on $\mathfrak{g}\times V'$. The affine quotient of $\mathfrak{g}\times V'$ by $G$ factors through the projection onto the first factor
\begin{eqnarray*}
    \begin{tikzcd}
    \mathfrak{g}\times V' \arrow{r} \arrow{d} & (\mathfrak{g}\times V')/\!\!/G \arrow[d,equal,"\rotatebox{90}{\(\sim\)}"]\\
    \mathfrak{g} \arrow{r} & \mathfrak{g}/\!\!/G
    \end{tikzcd}
\end{eqnarray*}
and defines a $G$-torsor over the open subvariety $(\mathfrak{g}\times V')_\mathrm{reg}\subset\mathfrak{g}\times V'$ consisting of all pairs $(X,v^*)$ such that $v^*,v^*X,\dots,v^*X^{n-1}$ form a basis of $V$.

\begin{proposition}
    \label{gl(n) regular integral proposition}
    The integral
    \begin{eqnarray}
        \label{gl(n) regular integral}
        &&
        \int_{G(F)\backslash G(\mathbb{A})} \sum_{(X,v^*)\in(\mathfrak{g}\times V')_\mathrm{reg}(F)} f\big(X\cdot\mathrm{ad}(g)\big) \Phi(v^*g)\big|\det(g)\big|^s \mathrm{d}g\\
        &=&
        \sum_{\mathfrak{o}_\mathrm{reg}} \Big(\int_{G(F)\backslash G(\mathbb{A})} \sum_{\begin{subarray}{c}(X,v^*)\in(\mathfrak{g}\times V')_\mathrm{reg}(F)\\X\in\mathfrak{o}_\mathrm{reg}\end{subarray}} f\big(X\cdot\mathrm{ad}(g)\big) \Phi(v^*g)\big|\det(g)\big|^s \mathrm{d}g\Big)\nonumber
    \end{eqnarray}
    converges absolutely if $\mathrm{Re}(s)>1$, and continues meromorphically over the entire complex plane, where the summation $\sum_{\mathfrak{o}_\mathrm{reg}}$ ranges over all regular conjugacy classes $\mathfrak{o}_\mathrm{reg}\subset\mathfrak{g}(F)$.
\end{proposition}

\begin{proof}
    The action of $G(F)$ on $(\mathfrak{g}\times V')_\mathrm{reg}(F)$ is faithful. By the theory of companion matrices, each orbit contains a unique representative of the form
    \begin{eqnarray*}
        \bigg(X_{a_1,\dots,a_n}=\left[\begin{matrix}
        a_1&\cdots&a_{n-1}&a_n\\
        1&\cdots&0&0\\
        \vdots&\ddots&\vdots&\vdots\\
        0&\cdots&1&0
        \end{matrix}\right], e_n^*=[0,\dots,0,1]\bigg)
    \end{eqnarray*}
    where $t^n-a_1t^{n-1}-\dots-a_n\in F[t]$ is the characteristic polynomial of $\mathfrak{o}_\mathrm{reg}$. Hence the integral (\ref{gl(n) regular integral}) is equal to
    \begin{eqnarray}
        \label{gl(n) regular integral coordinates}
        &&
        \int_{G(\mathbb{A})} \sum_{a_1,\dots,a_n\in F} f\big(X_{a_1,\dots,a_n}\cdot\mathrm{ad}(g)\big) \Phi(e_n^*g) \big|\det(g)\big|^s\mathrm{d}g\\
        &=&
        \int_{K}\Big(\int_{U(\mathbb{A})}\int_{(\mathbb{A}^\times)^{n-1}} \sum_{a_1,\dots,a_n\in F}
        f\big(X_{a_1,\dots,a_n}\cdot\mathrm{ad}(\left[\begin{smallmatrix}
        t&0\\0&1
        \end{smallmatrix}\right]uk)\big) \big|\det(t)\big|^s\mathrm{d}t\mathrm{d}u\Big)\times\nonumber\\
        &&
        \qquad\times
        \Big(\int_{Z(\mathbb{A})}
        \Phi(e_n^*zk)\big|\det(z)\big|^s\mathrm{d}z\Big)\mathrm{d}k\nonumber
    \end{eqnarray}
    by the Iwasawa decomposition $g=z\left[\begin{smallmatrix}t&0\\0&1\end{smallmatrix}\right]uk$ where $z\in Z(\mathbb{A})$, $t\in\mathrm{GL}_{n-1}(\mathbb{A})$ is diagonal, $u\in U(\mathbb{A})$ and $k\in K$.
    
    Choosing coordinates
    \begin{eqnarray*}
        t
        &=&
        \left[\begin{matrix}
        t^\circ&0\\0&1
        \end{matrix}\right]w
    \end{eqnarray*}
    where $t^\circ\in\mathrm{GL}_{n-2}(\mathbb{A})$ is diagonal and $w\in\mathbb{A}^\times$, and
    \begin{eqnarray*}
        u
        &=&
        \left[\begin{matrix}
        u^\circ&0\\0&1
        \end{matrix}\right]\left[\begin{matrix}
        I_{n-1}&v\\0&1
        \end{matrix}\right]
    \end{eqnarray*}
    where $u^\circ\in U_{n-1}(\mathbb{A})$, $I_i$ denotes the $i\times i$ identity matrix and $v\in\mathbb{A}^{n-1}$, then
    \begin{eqnarray*}
        && X_{a_1,\dots,a_n}\cdot\mathrm{ad}(\left[\begin{smallmatrix}
        t&0\\0&1
        \end{smallmatrix}\right]u)\\
        &=&
        \left[\begin{matrix}
        X_{a_1,\dots,a_{n-1}}&a_ne_1\\e_{n-1}^*&0
        \end{matrix}\right]\cdot\mathrm{ad}(\left[\begin{smallmatrix}
        wI_{n-1}&0\\0&1
        \end{smallmatrix}\right]\left[\begin{smallmatrix}
        t^\circ&0\\0&I_2
        \end{smallmatrix}\right]\left[\begin{smallmatrix}
        u^\circ&0\\0&1
        \end{smallmatrix}\right]\left[\begin{smallmatrix}
        I_{n-1}&v\\0&1
        \end{smallmatrix}\right])\\
        &=&
        \left[\begin{matrix}
        X_{a_1,\dots,a_{n-1}}\cdot\mathrm{ad}(\left[\begin{smallmatrix}
        t^\circ&0\\0&1
        \end{smallmatrix}\right]u^\circ)-wve_n^*&\vdots\\we_{n-1}^*&we_{n-1}^*v
        \end{matrix}\right]
    \end{eqnarray*}
    where the last column is equal to the transpose of
    \begin{eqnarray*}
        \left[\begin{matrix}
        \displaystyle\frac{a_n}{wt_1^\circ}
        &\displaystyle\frac{t_1^\circ v_1}{t_2^\circ}
        &\displaystyle\frac{t_2^\circ v_2}{t_3^\circ}
        &\dots
        &\displaystyle\frac{t_{n-3}^\circ v_{n-3}}{t_{n-2}^\circ}
        &\displaystyle\frac{t_{n-2}^\circ v_{n-2}}{1}
        &wv_{n-1}
        \end{matrix}\right]
    \end{eqnarray*}
    upto a unimodular change of variables of $v_1,\dots,v_{n-1}$, where $v_i$ denotes the $i$th entry of $v$ for $1\leq i\leq n-2$ and $t_j^\circ$ denotes the $j$th diagonal entry of $t^\circ$ for $1\leq j\leq n-1$.
    Hence the first factor of the inner integral in (\ref{gl(n) regular integral coordinates}) is equal to
    \begin{eqnarray*}
        && \int_{U(\mathbb{A})}\int_{(\mathbb{A}^\times)^{n-1}}\sum_{a_1,\dots,a_n\in F}
        f\big(X_{a_1,\dots,a_n}\cdot\mathrm{ad}(\left[\begin{smallmatrix}
        t&0\\0&1
        \end{smallmatrix}\right]uk)\big) \big|\det(t)\big|^s\mathrm{d}t\mathrm{d}u\\
        &=&
        \int\!\!\!...\!\!\!\int_\mathbb{A}\int_{\mathbb{A}^\times}\sum_{a_n\in F} \Big(\int_{U_{n-1}(\mathbb{A})}\int_{(\mathbb{A}^\times)^{n-2}}\nonumber\\
        &&\quad\times \sum_{a_1,\dots,a_{n-1}\in F}
        f\big(\bigg[\begin{smallmatrix}
        X_{a_1,\dots,a_{n-1}}\cdot\mathrm{ad}(\left[\begin{smallmatrix}
        t^\circ&0\\0&1
        \end{smallmatrix}\right]u^\circ)-wve_n^*&\vdots\\we_{n-1}^*&we_{n-1}^*v
        \end{smallmatrix}\bigg]\cdot\mathrm{ad}(k)\big)\times\\
        &&\qquad\times\big|\det(t^\circ)\big|^s\mathrm{d}t^\circ\mathrm{d}u^\circ\Big)|w|^{(n-1)s}\mathrm{d}^\times w\prod_{i=1}^{n-1}\mathrm{d}v_i
    \end{eqnarray*}
    which converges absolutely for $\mathrm{Re}(s)>1$ and continues meromorphically over $\mathbb{C}$ by induction on $n$.
    
    The second factor of the inner integral in (\ref{gl(n) regular integral coordinates}) is equal to
    \begin{eqnarray*}
        \int_{\mathbb{A}^\times}\Phi\big([0,\dots,0,z]k\big)\big|z\big|^{ns}\mathrm{d}^\times z
    \end{eqnarray*}
    which converges absolutely for $\mathrm{Re}(s)>1$ and continues meromorphically over $\mathbb{C}$, hence so does the entire integral (\ref{gl(n) regular integral}).
\end{proof}

\begin{proposition}
     \label{gl(n) regular orbital integral proposition}
     If $\mathfrak{o}_\mathrm{reg}$ has characteristic polynomial $p(t)=\prod_iq_i(t)^{m_i}$ where $q_i(t)\in F[t]$ are all distinct and irreducible, then
    \begin{eqnarray*}
        I_{\mathfrak{o}_\mathrm{reg}}(s)
        &=&
        \int_{G(F)\backslash G(\mathbb{A})} \sum_{\begin{subarray}{c}(X,v^*)\in(\mathfrak{g}\times V')_\mathrm{reg}(F)\\X\in\mathfrak{o}_\mathrm{reg}\end{subarray}} f\big(X\cdot\mathrm{ad}(g)\big) \Phi(v^*g)\big|\det(g)\big|^s \mathrm{d}g
    \end{eqnarray*}
    is the product of $\prod_i\zeta_{E_i}(s)\zeta_{E_i}(2s-1)\dots\zeta_{E_i}(m_is-m_i+1)$ by an entire function in $s$, where $E_i\simeq F[t]/(q_i)$.
\end{proposition}

\begin{proof}
    The orbit $(\mathfrak{o}_\mathrm{reg}\times V')_\mathrm{reg}$ contains a representative of the form
    \begin{eqnarray*}
        \bigg(X_{\mathfrak{o}_\mathrm{reg}}=\left[\begin{matrix}
        X_1&0&\cdots\\
        0&X_2&\cdots\\
        0&0&\ddots
        \end{matrix}\right], v_{\mathfrak{o}_\mathrm{reg}}^*=[e_{d_1m_1}^*,e_{d_2m_2}^*,\dots]\bigg)
    \end{eqnarray*}
    where $X_i$ is the $d_im_i\times d_im_i$ companion matrix with characteristic polynomial $q_i(t)^{m_i}$ with $\deg(q_i)=d_i$ and $e_m^*=[0,\dots,0,1]\in F^m$. Since $G(F)$ acts faithfully, the integral $I_{\mathfrak{o}_\mathrm{reg}}(s)$ is equal to
    \begin{eqnarray*}
        &&\int_{G(\mathbb{A})} f\big(X_{\mathfrak{o}_\mathrm{reg}}\cdot\mathrm{ad}(g)\big)\Phi(v_{\mathfrak{o}_\mathrm{reg}}^*g) \big|\det(g)\big|^s\mathrm{d}g\\
        &=&\int_{\prod_i\mathrm{GL}_{d_im_i}(\mathbb{A})\backslash G(\mathbb{A})} \Big(\int\!\!\!...\!\!\!\int_{\mathrm{GL}_{d_im_i}(\mathbb{A})} f\big(\left[\begin{smallmatrix}
        X_1\cdot\mathrm{ad}(g_1)&0&\dots\\
        0&X_2\cdot\mathrm{ad}(g_2)&\dots\\
        0&0&\dots
        \end{smallmatrix}\right]\cdot\mathrm{ad}(h)\big)\times\\
        &&\quad\times\Phi\big([e_{d_1m_1}^*g_1,e_{d_2m_2}^*g_2,\dots]h\big) \prod_i\big|\det(g_i)\big|^s\mathrm{d}g_i\Big)\mathrm{d}h,
    \end{eqnarray*}
    hence it suffices to show that the analogous integral
    \begin{eqnarray*}
        I_i(s)
        &=&
        \int_{\mathrm{GL}_{d_im_i}(\mathbb{A})} f_i\big(X_i\cdot\mathrm{ad}(g_i)\big) \Phi_i(e_{d_im_i}^*g_i)\big|\det(g_i)\big|^s\mathrm{d}g_i,
    \end{eqnarray*}
    where $f_i$ and $\Phi_i$ are Schwartz functions over $\mathfrak{gl}_{m_id_i}(\mathbb{A})$ and $\mathbb{A}^{m_id_i}$ respectively, is an entire multiple of $\zeta_{E_i}(s)\zeta_{E_i}(2s-1)\dots\zeta_{E_i}(m_is-m_i+1)$.
    
    The $d_im_i\times d_im_i$ matrix $X_i$ is conjugate to the $m_i\times m_i$ Jordan block matrix
    \begin{eqnarray*}
        \left[\begin{matrix}
        X_{q_i}&0&\cdots&0&0\\
        I_{d_i}&X_{q_i}&\cdots&0&0\\
        \vdots&\ddots&\ddots&\vdots&\vdots\\
        0&0&\ddots&X_{q_i}&0\\
        0&0&\cdots&I_{d_i}&X_{q_i}
        \end{matrix}\right]
    \end{eqnarray*}
    where $X_{q_i}$ is the $d_i\times d_i$ companion matrix with characteristic polynomial $q_i(t)$ and $I_{d_i}$ denotes the $d_i\times d_i$ identity matrix. The centralizer of $X_{q_i}$ in $\mathrm{GL}_{d_i}(F)$ is isomorphic to $E_i^\times$, hence under this identification $\mathbb{A}_{E_i}^{m_i}\simeq\mathbb{A}_F^{d_im_i}$, $\mathrm{GL}_{m_i}(\mathbb{A}_{E_i})\subset\mathrm{GL}_{d_im_i}(\mathbb{A}_F)$ and choosing a new representative
    \begin{eqnarray*}
        \bigg(X_{E_i}=\left[\begin{matrix}
        x_{E_i}&0_{E_i}&\cdots&0_{E_i}\\
        1_{E_i}&x_{E_i}&\cdots&0_{E_i}\\
        \vdots&\ddots&\ddots&\vdots\\
        0_{E_i}&0_{E_i}&\cdots&x_{E_i}
        \end{matrix}\right], e_{E_i}^*=[0_{E_i},\dots,0_{E_i},1_{E_i}]\bigg)
    \end{eqnarray*}
    where $x_{E_i}\in E_i$ corresponds to $X_{p_i}$ and $0_{E_i}$ and $1_{E_i}$ denotes the zero and unit element of $E_i$ respectively, the integral $I_i(s)$ is equal to
    \begin{eqnarray*}
        &&\int_{\mathrm{GL}_{m_i}(\mathbb{A}_{E_i})\backslash\mathrm{GL}_{d_im_i}(\mathbb{A}_F)}\\
        &&\quad\times \Big(\int_{\mathrm{GL}_{m_i}(\mathbb{A}_{E_i})} f_i\big(X_{E_i}\cdot\mathrm{ad}(g_ih_i)\big) \Phi_i(e_{E_i}^*g_ih_i)\big|\det(g_ih_i)\big|^s\mathrm{d}g_i\Big)\mathrm{d}h_i.
    \end{eqnarray*}
    Hence it suffices to analyze the analogous inner integral defined over $E_i$, which is treated in the next lemma.
\end{proof}

\begin{lemma}
    \label{gl(n) regular orbital integral lemma}
    Let $X_d=\left[\begin{smallmatrix}
    d&0&\dots\\
    1&d&\dots\\
    \dots&\dots&\dots
    \end{smallmatrix}\right]\in G(F)$ be a Jordan block with ones along the subdiagonal and $e_n^*=[0,\dots,0,1]\in V(F)$, then the integral
    \begin{eqnarray*}
        I_d(s)
        &=&
        \int_{G(\mathbb{A})} f\big(X_d\cdot\mathrm{ad}(g)\big)\Phi(e_n^*g) \big|\det(g)\big|^s\mathrm{d}g
    \end{eqnarray*}
    is the product of $\zeta_F(s)\zeta_F(2s-1)\dots\zeta_F(ns-n+1)$ by an entire function.
\end{lemma}

\begin{proof}
    By the Iwasawa decomposition it suffices to evaluate the analogous integral $I_{d,B^-}(s)$ which is integrated over $B^-(\mathbb{A})$ instead of $G(\mathbb{A})$. Choosing coordinates
    \begin{eqnarray*}
        b
        &=&
        \left[\begin{matrix}
        z&0&\cdots&0&0&0\\
        u_{n-1}&z&\cdots&0&0&0\\
        u_{n-2}&u_{n-1}&\ddots&0&0&0\\
        \vdots&\ddots&\ddots&\ddots&\vdots&\vdots\\
        u_2&u_3&\ddots&u_{n-1}&z&0\\
        u_1&u_2&\cdots&u_{n-2}&u_{n-1}&z
        \end{matrix}\right]\left[\begin{matrix}
        b^\circ&0\\0&1
        \end{matrix}\right],
    \end{eqnarray*}
    where $z\in\mathbb{A}^\times$, $u_1,\dots,u_{n-1}\in\mathbb{A}$ and $b^\circ\in\mathrm{GL}_{n-1}(\mathbb{A})$ is lower triangular, then the integral $I_{d,B^-}(s)$ is equal to
    \begin{eqnarray*}
        &&
        \int\!\!\!...\!\!\!\int_\mathbb{A}\int_{\mathbb{A}^\times} \Big(\int_{B_{n-1}^-(\mathbb{A})} f\big(X_d\cdot\mathrm{ad}(\left[\begin{smallmatrix}
        z&0&\dots\\
        u_{n-1}&z&\dots\\
        \dots&\dots&\dots
        \end{smallmatrix}\right]\left[\begin{smallmatrix}
        b^\circ&0\\0&1
        \end{smallmatrix}\right])\big)\times\\
        &&
        \quad\times\Phi\big(e_n^*\left[\begin{smallmatrix}
        z&0&\dots\\
        u_{n-1}&z&\dots\\
        \dots&\dots&\dots
        \end{smallmatrix}\right]\left[\begin{smallmatrix}
        b^\circ&0\\0&1
        \end{smallmatrix}\right]\big) \big|\det(b^\circ)\big|^s\big|\det(b^\circ)\big|\mathrm{d}b^\circ\Big)\big|z\big|^{ns}\mathrm{d}^\times z\prod_{i=1}^{n-1}\frac{\mathrm{d}u_i}{|z|}\\
        &=&
        \int\!\!\!...\!\!\!\int_\mathbb{A}\int_{\mathbb{A}^\times} \Big(\int_{B_{n-1}^-(\mathbb{A})} f\big(X_d\cdot\mathrm{ad}(\left[\begin{smallmatrix}
        b^\circ&0\\0&1
        \end{smallmatrix}\right])\big)\times\\
        &&
        \quad\times\Phi\big([u_1,\dots,u_{n-1},z]\left[\begin{smallmatrix}
        b^\circ&0\\0&1
        \end{smallmatrix}\right]\big) \big|\det(b^\circ)\big|^{s+1}\mathrm{d}b^\circ\Big)\big|z\big|^{ns-n+1}\mathrm{d}^\times z\prod_{i=1}^{n-1}\mathrm{d}u_i\\
        &=&
        \int\!\!\!...\!\!\!\int_\mathbb{A}\int_{\mathbb{A}^\times} \Big(\int_{B_{n-1}^-(\mathbb{A})} f\big(\left[\begin{smallmatrix}
        X_d^\circ\cdot\mathrm{ad}(b^\circ)&0\\e_{n-1}^*b^\circ&d
        \end{smallmatrix}\right]\big)\big|\det(b^\circ)\big|^s\mathrm{d}b^\circ\Big)\times\\
        &&
        \quad\times\Phi\big([v_1,\dots,v_{n-1},z]\big) \big|z\big|^{ns-n+1}\mathrm{d}^\times z\prod_{i=1}^{n-1}\mathrm{d}v_i
    \end{eqnarray*}
    where $X_d^\circ$ denotes the analogous $(n-1)\times(n-1)$ lower triangular Jordan block and $[v_1,\dots,v_{n-1}]=[u_1,\dots,u_{n-1}]b^\circ$. Hence $I_{d,B^-}(s)$ is the product of $I_{d,B_{n-1}^-}(s)$ by a Tate integral equal to an entire multiple of $\zeta_F(ns-n+1)$, and the lemma follows by induction on $n$.
\end{proof}

\section{Cancellation of singularities}
\label{sec:4}

\paragraph{A mirabolic analogue of the parabolic descent operator} Let $P\subset G$ be the parabolic subgroup associated to a partial flag $V=V_r\supset\dots\supset V_0=0$ with successive subquotients $W_i=V_i/V_{i-1}$. Let $\widetilde{f}$ denote the partial Fourier transform of $f$ defined by
\begin{eqnarray*}
    \widetilde{f}(X,U_r^*,\dots,U_2^*)
    &=&
    \int_{\mathfrak{u}_P(\mathbb{A})} f(\tilde{X}+U) \prod_{i=2}^r\overline{\psi}\big(\mathrm{tr}(U_i^*U)\big)\mathrm{d}U
\end{eqnarray*}
where $U_i^*:W_{i-1}(\mathbb{A})\to W_i(\mathbb{A})$ commutes with $X$ and $\tilde{X}$ is a lift of $X$ in $\mathfrak{p}(\mathbb{A})$. In general $\widetilde{f}$ is defined upto a multiplicative constant depending on the choice of the lift $\tilde{X}$. However if $P$ is a standard parabolic subgroup, then $\widetilde{f}$ is uniquely defined by choosing $\tilde{X}\in\mathfrak{m}_P$.

Then the parabolic descent $(f,\Phi)_P\in\mathcal{S}(\mathfrak{m}_P(\mathbb{A})\times W_r(\mathbb{A})\times\dots\times W_1(\mathbb{A}))$ is defined by
\begin{eqnarray*}
    (f,\Phi)_P(X,w_r^*,\dots,w_1^*)
    &=&
    \int_K\int\!\!\!...\!\!\!\int \widetilde{f^k}(X,U_r^*,\dots,U_2^*) \prod_{i=2}^r\mathrm{d}U_i^*\Phi(w_1^*k)\mathrm{d}k
\end{eqnarray*}
where $f^k$ denotes the composite function $f\circ\mathrm{ad}(k)$ and the inner integral is integrated over the affine space consisting of all $U_i^*\in\mathrm{Hom}(W_{i-1}(\mathbb{A}),W_i(\mathbb{A}))$ such that $w_i^*=w_{i-1}^*U_i^*$. In particular,
\begin{itemize}
    \item if $X\in\mathfrak{m}_P(\mathbb{A})$ is regular semisimple which implies that $U_i^*=0$ for all $i=2,\dots,r$, then \begin{eqnarray*}
        (f,\Phi)_P(X,0,\dots,0,w_1^*)
        &=&
        f_P(X)\Phi(w_1^*);
    \end{eqnarray*}
    \item if $X\in\mathfrak{m}_P(\mathbb{A})$, $w_i^*\in W_i(\mathbb{A})$ and $(X|_{W_i},w_i^*)\in(\mathfrak{gl}(W_i)\times W_i')_\mathrm{reg}$ for all $i=1,\dots,r$, then
    \begin{eqnarray*}
        (f,\Phi)_P(X,w_r^*,\dots,w_1^*)
        &=&
        \widetilde{f}(X,U_r^*,\dots,U_2^*)\Phi(w_1^*)
    \end{eqnarray*}
    where $U_i^*$ denotes the unique $\mathbb{A}[X]$-linear map from $W_{i-1}(\mathbb{A})$ to $W_i(\mathbb{A})$ such that $w_i^*=w_{i-1}^*U_i^*$ for all $i=2,\dots,r$.
\end{itemize}

\begin{definition}\label{regularized mirabolic integral definition}
    Define the regularized mirabolic integrals $I(s)$ and $I_{\mathfrak{o}_\lambda}(s)$, where $\lambda=[n_r,\dots,n_1]\vdash n$ and $\mathfrak{o}_\lambda\subset\mathfrak{g}(F)$ is a conjugacy class which has Frobenius normal form of partition type $\lambda$, by
    \begin{eqnarray*}
        I(s)
        &=&
        \int_{G(F)\backslash G(\mathbb{A})} \sum_P(-1)^{\ell(P)} K_P(g;f) \sum_{v^*\in(V^{U_P})'(F)} \Phi(v^*g)\big|\det(g)\big|^s \mathrm{d}g
    \end{eqnarray*}
    where the summation $\sum_P$ ranges over all parabolic subgroups $P\subset G$ defined over $F$ and $V^{U_P}=V_1\subset V$ denotes the subspace consisting of $U_P$-invariants, and
    \begin{eqnarray*}
        I_{\mathfrak{o}_\lambda}(s)
        &=&
        \int_{M_\lambda(F)\backslash M_\lambda(\mathbb{A})} \sideset{}{'} \sum_{\begin{subarray}{c}(X,w_r^*,\dots,w_1^*)\in(\mathfrak{m}_\lambda\times W_r'\times\dots\times W_1')(F)\\X\in \mathfrak{o}_\lambda\end{subarray}}\\
        &&\quad\times (f,\Phi)_{P_\lambda}\big(X\cdot\mathrm{ad}(h),w_r^*h,\dots,w_1^*h\big) \big|\det(h)\big|^s \mathrm{d}h
    \end{eqnarray*}
    where the primed summation ranges over all sequences $(X,w_r^*,\dots,w_1^*)$ such that $(X|_{W_i},w_i^*)\in(\mathfrak{gl}(W_i)\times W_i')_\mathrm{reg}$ for all $i=1,\dots,r$.
\end{definition}

\begin{proposition}
    \label{gl(n) singular integral proposition}
    The regularized mirabolic integrals
    \begin{eqnarray*}
        I(s)
        &=&
        \sum_{\lambda\vdash n} \sum_{\mathfrak{o}_\lambda} I_{\mathfrak{o}_\lambda}(s)
    \end{eqnarray*}
    converge absolutely if $\mathrm{Re}(s)>1$, and continue meromorphically over the entire complex plane, where the summation $\sum_{\mathfrak{o}_\lambda}$ ranges over all conjugacy classes $\mathfrak{o}_\lambda\subset\mathfrak{g}(F)$ which have Frobenius normal form of partition type $\lambda$.
\end{proposition}

\begin{proof}
    Define
    \begin{eqnarray*}
        I_\lambda(s)
        &=&
        \int_{M_\lambda(F)\backslash M_\lambda(\mathbb{A})} \sideset{}{'} \sum_{\begin{subarray}{c}X\in\mathfrak{m}_\lambda(F)\\w_i^*\in W_i'(F)\end{subarray}} (f,\Phi)_{P_\lambda}\big(X\cdot\mathrm{ad}(h),\dots,w_1^*h\big) \big|\det(h)\big|^s \mathrm{d}h
    \end{eqnarray*}
    where the primed summation ranges over all $(X|_{W_i},w_i^*)\in(\mathfrak{gl}(W_i)\times W_i')_\mathrm{reg}$, then by Proposition \ref{gl(n) regular integral proposition}
    \begin{eqnarray*}
        I_\lambda(s)
        &=&
        \sum_{\mathfrak{o}_\lambda} I_{\mathfrak{o}_\lambda}(s)
    \end{eqnarray*}
    converges absolutely if $\mathrm{Re}(s)>1$, and continues meromorphically over the entire complex plane. Hence it suffices to show that
    \begin{eqnarray*}
        I(s)
        &=&
        \sum_{\lambda\vdash n} I_\lambda(s)
    \end{eqnarray*}
    under the assumption that $\mathrm{Re}(s)>1$ without further convergence issues.
    
    Further decomposing
    \begin{eqnarray*}
        I(s)
        &=&
        \sum_{m=1}^n I_m(s)
    \end{eqnarray*}
    by the dimension $m=\dim W$ where $W\subset V$ is the subspace spanned by $v^*,v^*X,v^*X^2,\dots$, it suffices to consider a summand
    \begin{eqnarray}
        \label{gl(n) singular integral summand m}
        I_m(s)
        &=&
        \int_{P_{[n-m,m]}(F)\backslash G(\mathbb{A})} \sum_P (-1)^{\ell(P)} \sideset{}{'} \sum_{\begin{subarray}{c}X\in\mathfrak{p}_{[n-m,m]}(F)\cap\mathfrak{p}(F)\\w^*\in(0^{n-m}\times F^m)'\cap(V^{U_P})'(F)\end{subarray}}\\
        &&\quad\times \int_{\mathfrak{u}_P(\mathbb{A})/\mathfrak{u}_P(F)} f\big((X+U)\cdot\mathrm{ad}(g)\big)\mathrm{d}U \Phi(w^*g)\big|\det(g)\big|^s \mathrm{d}g\nonumber
    \end{eqnarray}
    where the primed summation ranges over all $(X=\left[\begin{smallmatrix}
    X^\circ&Y\\0&X_1
    \end{smallmatrix}\right],w^*=[0,w_1^*])$ such that $(X_1,w_1^*)\in(\mathfrak{gl}_m(F)\times F^m{}')_\mathrm{reg}$.
    
    Since $(X_1,w_1^*)$ is regular, all parabolic subgroups $P$ appearing in (\ref{gl(n) singular integral summand m}) such that $(X,w^*)\in(\mathfrak{p}\times(V^{U_P})')(F)$ must contain $\mathrm{GL}(m)$ in the lower right corner and are hence either of the form
    \begin{eqnarray*}
        P^\circ_+
        &=&
        \left[\begin{matrix}
        P^\circ&\mathrm{Hom}(F^{n-m},F^m)\\
        \mathrm{Hom}(F^m,W_2)&\mathrm{GL}(m)
        \end{matrix}\right]
    \end{eqnarray*}
    with unipotent radical $U_{P_+^\circ}\simeq U_{P^\circ}\times\mathrm{Hom}(F^{n-m}/W_2,F^m)$ and partial flag length $\ell(P_+^\circ)=\ell(P^\circ)$, or of the form
    \begin{eqnarray*}
        P^\circ_-
        &=&
        \left[\begin{matrix}
        P^\circ&\mathrm{Hom}(F^{n-m},F^m)\\
        0&\mathrm{GL}(m)
        \end{matrix}\right]
    \end{eqnarray*}
    with unipotent radical $U_{P_-^\circ}\simeq U_{P^\circ}\times\mathrm{Hom}(F^{n-m},F^m)$ and partial flag length $\ell(P_-^\circ)=\ell(P^\circ)+1$, where $P^\circ\subset\mathrm{GL}(n-m)$ is a parabolic subgroup which stabilizes the subspace $W_2=(F^{n-m})^{U_{P^\circ}}\subset F^{n-m}$. Hence $I_m(s)$ is equal to
    \begin{eqnarray*}
        &&
        \int_{P_{[n-m,m]}(F)\backslash G(\mathbb{A})} \sum_{P^\circ\subset\mathrm{GL}(n-m)} (-1)^{\ell(P^\circ)} \sum_{\begin{subarray}{c}X^\circ\in\mathfrak{p}^\circ(F),Y\in\mathrm{Hom}(F^{n-m},F^m)\\(X_1,w_1^*)\in(\mathfrak{gl}_m(F)\times F^m{}')_\mathrm{reg}\end{subarray}}\\
        &&\quad\times \int_{\mathfrak{u}_{P^\circ}(\mathbb{A})/\mathfrak{u}_{P^\circ}(F)} \Big(\int f\big(\left[\begin{smallmatrix}
        X^\circ+U^\circ&Y+U_+\\0&X_1
        \end{smallmatrix}\right]\cdot\mathrm{ad}(g)\big) \mathrm{d}U_++\\
        &&\qquad -\int f\big(\left[\begin{smallmatrix}
        X^\circ+U^\circ&Y+U_-\\0&X_1
        \end{smallmatrix}\right]\cdot\mathrm{ad}(g)\big) \mathrm{d}U_-\Big) \mathrm{d}U^\circ\Phi([0,w_1^*]g)\big|\det(g)\big|^s \mathrm{d}g
    \end{eqnarray*}
    where $U_+$ is integrated over $\mathrm{Hom}(\mathbb{A}^{n-m},\mathbb{A}^m)/\mathrm{Hom}(F^{n-m},F^m)$ and $U_-$ is integrated over $\mathrm{Hom}(\mathbb{A}^{n-m}/W_2(\mathbb{A}),\mathbb{A}^m)/\mathrm{Hom}(F^{n-m}/W_2(F),F^m)$.
    
    Let $V\in\mathrm{Hom}(\mathbb{A}^{n-m},\mathbb{A}^m)$, then
    \begin{eqnarray*}
        &&
        \left[\begin{matrix}
        X^\circ+U^\circ&(Y+U_\pm)\\0&X_1
        \end{matrix}\right] \cdot\mathrm{ad}\Big(\left[\begin{matrix}1&V\\0&1\end{matrix}\right]\Big)\\
        &=&
        \left[\begin{matrix}
        X^\circ+U^\circ&(Y+U_\pm)+(X^\circ+U^\circ)V-VX_1\\0&X_1
        \end{matrix}\right],
    \end{eqnarray*}
    hence
    \begin{eqnarray*}
        &&
        \int_{\mathrm{Hom}(\mathbb{A}^{n-m},\mathbb{A}^m)/\mathrm{Hom}(F^{n-m},F^m)} \sum_{Y\in\mathrm{Hom}(F^{n-m},F^m)}\\
        &&\quad\times\Big(
        \int f\big(\left[\begin{smallmatrix}
        X^\circ+U^\circ&Y+U_+\\0&X_1
        \end{smallmatrix}\right]\cdot\mathrm{ad}(\left[\begin{smallmatrix}1&V\\0&1\end{smallmatrix}\right])\big) \mathrm{d}U_++\\
        &&\qquad-\int f\big(\left[\begin{smallmatrix}
        X^\circ+U^\circ&Y+U_-\\0&X_1
        \end{smallmatrix}\right]\cdot\mathrm{ad}(\left[\begin{smallmatrix}1&V\\0&1\end{smallmatrix}\right])\big) \mathrm{d}U_-\Big)\mathrm{d}V\\
        &=&
        \sum_{Y\in\mathrm{Hom}(W_2(F),F^m)_V} f_1(\left[\begin{smallmatrix}
        X^\circ+U^\circ&Y\\0&X_1
        \end{smallmatrix}\right]) -\int_{\mathrm{Hom}(W_2(\mathbb{A}),\mathbb{A}^m)_V} f_1(\left[\begin{smallmatrix}
        X^\circ+U^\circ&Y\\0&X_1
        \end{smallmatrix}\right])\mathrm{d}Y
    \end{eqnarray*}
    where $\mathrm{Hom}(W_2,F^m)_V$ denotes the quotient of $\mathrm{Hom}(W_2,F^m)$ consisting of $\mathrm{ad}(\left[\begin{smallmatrix}1&V\\0&1\end{smallmatrix}\right])$-coinvariants and $f_1$ denotes the fiberwise integral of $f$ along the fibers of the composite map
    \begin{eqnarray*}
        \begin{tikzcd}
        \mathrm{Hom}(\mathbb{A}^{n-m},\mathbb{A}^m) \arrow{r} \arrow{ddr} & \mathrm{Hom}(\mathbb{A}^{n-m},\mathbb{A}^m)\Big/\mathrm{Hom}(\mathbb{A}^{n-m}/W_2(\mathbb{A}),\mathbb{A}^m) \arrow[d,equal,"\rotatebox{90}{\(\sim\)}"]\\
        & \mathrm{Hom}(W_2(\mathbb{A}),\mathbb{A}^m) \arrow{d}\\
        & \mathrm{Hom}(W_2(\mathbb{A}),\mathbb{A}^m)_V,
        \end{tikzcd}
    \end{eqnarray*}
    which is equal to
    \begin{eqnarray*}
        \sum_{\begin{subarray}{c}Y^*\in\mathrm{Hom}(F^m,W_2(F))\\X_1Y^*=Y^*X^\circ|_{W_2}\end{subarray}} \widetilde{f}(\left[\begin{smallmatrix}
        X^\circ+U^\circ&0\\0&X_1
        \end{smallmatrix}\right],Y^*) -\widetilde{f}(\left[\begin{smallmatrix}
        X^\circ+U^\circ&0\\0&X_1
        \end{smallmatrix}\right],0)
    \end{eqnarray*}
    by Poisson summation where $\widetilde{f}$ denotes the partial Fourier transform of $f_1$ defined by
    \begin{eqnarray*}
        \widetilde{f}(\left[\begin{smallmatrix}
        X^\circ+U^\circ&0\\0&X_1
        \end{smallmatrix}\right],Y^*)
        &=&
        \int_{\mathrm{Hom}(W_2(\mathbb{A}),\mathbb{A}^m)_V} f_1(\left[\begin{smallmatrix}
        X^\circ+U^\circ&Y\\0&X_1
        \end{smallmatrix}\right]) \overline{\psi}\big(\mathrm{tr}(Y^*Y)\big)\mathrm{d}Y
    \end{eqnarray*}
    where $Y^*\in\mathrm{Hom}(\mathbb{A}^m,W_2(\mathbb{A}))^V$ intertwines the actions of $X_1$ and $X^\circ$.
    
    Hence by the partial Iwasawa decomposition $g=\left[\begin{smallmatrix}1&V\\0&1\end{smallmatrix}\right]\left[\begin{smallmatrix}g^\circ&0\\0&g_1\end{smallmatrix}\right]k$ where $V\in\mathrm{Hom}(\mathbb{A}^{n-m},\mathbb{A}^m)/\mathrm{Hom}(F^{n-m},F^m)$, $g^\circ\in\mathrm{GL}_{n-m}(F)\backslash\mathrm{GL}_{n-m}(\mathbb{A})$, and $g_1\in\mathrm{GL}_m(F)\backslash\mathrm{GL}_m(\mathbb{A})$, $k\in K$, the integral $I_m(s)$ is equal to
    \begin{eqnarray*}
        &&
        \int_{\mathrm{GL}_m(F)\backslash\mathrm{GL}_m(\mathbb{A})}  \sum_{(X_1,w_1^*)\in(\mathfrak{gl}_m(F)\times F^m{}')_\mathrm{reg}} \int_K\\
        &&\quad\times\Big( \int_{\mathrm{GL}_{n-m}(F)\backslash\mathrm{GL}_{n-m}(\mathbb{A})} \sum_{P^\circ\subset\mathrm{GL}(n-m)} (-1)^{\ell(P^\circ)} \sum_{X^\circ\in\mathfrak{p}^\circ(F)} \int_{\mathfrak{u}_{P^\circ}(\mathbb{A})/\mathfrak{u}_{P^\circ}(F)}\\
        &&\qquad\times \sum_{\begin{subarray}{c}w_2^*\in W_2(F)\\w_2^*\neq0\end{subarray}} \widetilde{f^k}(\left[\begin{smallmatrix}
        (X^\circ+U^\circ)\cdot\mathrm{ad}(g^\circ)&0\\0&X_1\cdot\mathrm{ad}(g_1)
        \end{smallmatrix}\right],g_1^{-1}Y^*g^\circ) \frac{|\det(g^\circ)|^m}{|\det(g_1)|^{n-m}} \times\\
        &&\quad\qquad\times \mathrm{d}U^\circ\big|\det(g^\circ)\big|^s\mathrm{d}g^\circ\Big)\Phi([0,w_1^*g_1]k)\mathrm{d}k \big|\det(g_1)\big|^s \frac{|\det(g_1)|^{n-m}}{|\det(g^\circ)|^m}\mathrm{d}g_1,
    \end{eqnarray*}
    where $Y^*$ denotes the unique $F[X]$-linear map from $F^m$ to $W_2(F)$ such that $w_2^*=w_1^*Y^*$, which is equal to
    \begin{eqnarray*}
        &&
        \int_{\mathrm{GL}_m(F)\backslash\mathrm{GL}_m(\mathbb{A})}  \sum_{(X_1,w_1^*)\in(\mathfrak{gl}_m(F)\times F^m{}')_\mathrm{reg}} \Big(\sum_{\lambda^\circ\vdash(n-m)} \int_{M_{\lambda^\circ}(F)\backslash M_{\lambda^\circ}(\mathbb{A})}\\
        &&\quad\times \sideset{}{'} \sum_{\begin{subarray}{c}X^\circ\in\mathfrak{m}_{\lambda^\circ}(F)\\w_i^*\in W_i'(F)\end{subarray}} (f,\Phi)_{P_{[\lambda^\circ,m]}}(\left[\begin{smallmatrix}
        X^\circ\cdot\mathrm{ad}(h^\circ)&0\\0&X_1\cdot\mathrm{ad}(g_1)
        \end{smallmatrix}\right],w_r^*h^\circ,\dots,w_1^*g_1)\times\\
        &&\qquad\times \big|\det(h^\circ)\big|^s \mathrm{d}h^\circ\Big) \big|\det(g_1)\big|^s \mathrm{d}g_1\\
        &=&
        \sum_{\begin{subarray}{c}
        \lambda\vdash n\\\lambda=[\lambda^\circ,m]\end{subarray}} I_\lambda(s)
    \end{eqnarray*}
    by induction on $n$.
\end{proof}

\begin{proposition}
     \label{gl(n) singular orbital integral proposition}
     If $\mathfrak{o}_\lambda$ has invariant factors $p_r\mid p_{r-1}\mid\dots\mid p_1$ where $p_j(t)=\prod_iq_i(t)^{m_{ij}}$ and $q_i(t)\in F[t]$ are all distinct and irreducible, then the regularized mirabolic integral $I_{\mathfrak{o}_\lambda}(s)$ is the product of
    \begin{eqnarray*}
        \prod_{i,j}\prod_{k=1}^{m_{ij}} \zeta_{E_i}\big(h_{\lambda_i}(j,k)s-m_{ij}+k\big)
    \end{eqnarray*}
    by an entire function in $s$, where $E_i\simeq F[t]/(q_i)$ and
    \begin{eqnarray*}
        h_{\lambda_i}(j,k)
        &=&
        m_{ij}-k+\big|\{l\geq j\mid m_{il}\geq k\}\big|
    \end{eqnarray*}
    where $|S|$ denotes the cardinality of a finite set $S$.
\end{proposition}

\begin{proof}
    By the same argument as in the proof of Proposition \ref{gl(n) regular orbital integral proposition}, it suffices to consider the special case when $f_j=(t-d)^{n_j}$ where $d\in F$ is the unique eigenvalue of $\mathfrak{o}_\lambda$, and proceed as in the proof of Lemma \ref{gl(n) regular orbital integral lemma}.
    
    Choosing the representative
    \begin{eqnarray*}
        \bigg(X=\left[\begin{matrix}
        X_r&0&\cdots&0\\
        0&X_{r-1}&\cdots&0\\
        \vdots&\vdots&\ddots&\vdots\\
        0&0&\cdots&X_1
        \end{matrix}\right], w_r^*=e_{n_r}^*,\dots, w_1^*=e_{n_1}^*\bigg)
    \end{eqnarray*}
    where $X_j=\left[\begin{smallmatrix}
    d&0&\dots\\
    1&d&\dots\\
    \dots&\dots&\dots
    \end{smallmatrix}\right]\in\mathrm{GL}_{n_j}(F)$ is a Jordan block with ones along the subdiagonal and $e_{n_j}^*=[0,\dots,0,1]\in W_j'(F)\simeq F^{n_j}{}'$, it suffices to evaluate the analogous integral
    \begin{eqnarray*}
        I_{\mathfrak{o}_\lambda,B^-}(s)
        &=&
        \int_{B_{n_1}^-(\mathbb{A})}\dots\int_{B_{n_r}^-(\mathbb{A})}\\
        &&\quad\times
        (f,\Phi)_{P_\lambda}\big(\left[\begin{smallmatrix}
        X_r\cdot\mathrm{ad}(b_r)&0&\dots\\
        0&\dots&0\\
        \dots&0&X_1\cdot\mathrm{ad}(b_1)
        \end{smallmatrix}\right],w_r^*b_r,\dots,w_1^*b_1\big)\times\\
        &&\qquad\times\prod_{j=1}^r\big|\det(b_j)\big|^s \mathrm{d}b_j
    \end{eqnarray*}
    integrated over $B_{n_j}^-(\mathbb{A})\subset\mathrm{GL}_{n_j}(\mathbb{A})$.
    
    Choosing coordinates
    \begin{eqnarray*}
        b_j
        &=&
        b_{j,n_j}\left[\begin{matrix}
        b_{j,n_j-1}&0\\0&1
        \end{matrix}\right]\dots\left[\begin{matrix}
        b_{j,2}&0\\0&I_{n_j-2}
        \end{matrix}\right]\left[\begin{matrix}
        b_{j,1}&0\\0&I_{n_j-1}
        \end{matrix}\right]
    \end{eqnarray*}
    as in the proof of Lemma \ref{gl(n) regular orbital integral lemma} where each $b_{j,k}\in B_k^-(\mathbb{A})$ is of the form
    \begin{eqnarray*}
        b_{j,k}
        &=&
        \left[\begin{matrix}
        z&0&\cdots&0&0&0\\
        u_{k-1}&z&\cdots&0&0&0\\
        u_{k-2}&u_{k-1}&\ddots&0&0&0\\
        \vdots&\ddots&\ddots&\ddots&\vdots&\vdots\\
        u_2&u_3&\ddots&u_{k-1}&z&0\\
        u_1&u_2&\cdots&u_{k-2}&u_{k-1}&z
        \end{matrix}\right]
    \end{eqnarray*}
    where $z\in\mathbb{A}^\times$, $u_1,\dots,u_{k-1}\in\mathbb{A}$ and $I_l$ denotes the $l\times l$ identity matrix, then by a similar computation, the $\mathrm{d}b_{1,k}$ integral produces a Tate integral of the form
    \begin{eqnarray}
        \label{gl(n) lambda tate integral}
        &&\int\!\!\!...\!\!\!\int_\mathbb{A}\int_{\mathbb{A}^\times} \int(f,\Phi)_{P_\lambda}\big(\dots,X_2\cdot\mathrm{ad}(b_2b_{1,k}^\circ),X_{1,k}\cdot\mathrm{ad}(b_{1,k}),\dots\big)\times\nonumber\\
        &&\quad\times\big|\det(b_2b_{1,k}^\circ)\big|^s\big|z\big|^{ks-k+1}\mathrm{d}^\times z\prod_{i=1}^{k-1}\mathrm{d}v_i
    \end{eqnarray}
    where $b_{1,k}^\circ\in\mathrm{GL}_{n_2}(\mathbb{A})$ denotes the action of $b_{1,k}$ on $W_2(\mathbb{A})$ induced from the unique $\mathbb{A}[X]$-linear map from $W_1(\mathbb{A})$ to $W_2(\mathbb{A})$ which maps $e_{n_1}^*$ to $e_{n_2}^*$. There are two cases:
    \begin{itemize}
        \item if $k\leq n_1-n_2$, then $b_{1,k}^\circ=I_{n_2}$ and the Tate integral (\ref{gl(n) lambda tate integral}) is an entire multiple of \begin{eqnarray*}
            \zeta_F(ks-k+1)
            &=&
            \zeta_F\big(h_\lambda(1,l)s-n_1+l\big)
        \end{eqnarray*}
        where $l=n_1-k+1$;
        \item if $k>n_1-n_2$, then
        \begin{eqnarray*}
            b_{1,k}^\circ
            &=&
            \left[\begin{matrix}
            z&\cdots&0&0\\
            \vdots&\ddots&\vdots&\vdots\\
            u_{n_1-n_2+1}&\dots&z&0\\
            0&\cdots&0&I_{n_1-k}
            \end{matrix}\right]
        \end{eqnarray*}
        and the Tate integral (\ref{gl(n) lambda tate integral}) is an entire multiple of
        \begin{eqnarray*}
            \zeta_F\big(h_\lambda(2,l)s+ks-k+1\big)
            &=&
            \zeta_F\big(h_\lambda(1,l)s-n_1+l\big)
        \end{eqnarray*}
        by induction, where $l=n_1-k+1$.
    \end{itemize}
\end{proof}

\begin{theorem}[Mirabolic trace formula for $\mathfrak{gl}(n)$]
    \label{mirabolic trace formula}
    Let $\Phi\in\mathcal{S}(V(\mathbb{A}))$, then the distribution $I(s)\in\mathcal{S}'(\mathfrak{g}(\mathbb{A}))$ defined by
    \begin{eqnarray*}
        I(s)
        &=&
        \sum_{\lambda\vdash n} \sum_{\mathfrak{o}_\lambda} I_{\mathfrak{o}_\lambda}(s)
    \end{eqnarray*}
    where the summation $\sum_{\mathfrak{o}_\lambda}$ ranges over all conjugacy classes $\mathfrak{o}_\lambda\subset\mathfrak{g}(F)$ which have Frobenius normal form of partition type $\lambda$
    \begin{itemize}
        \item converges absolutely if $\mathrm{Re}(s)>1$ and continues meromorphically to a sum of entire multiples of zeta functions, and
        \item is invariant under the Fourier transform $f\leftrightarrow\widehat{f}$ on $\mathcal{S}(\mathfrak{g}(\mathbb{A}))$.
    \end{itemize}
\end{theorem}

\begin{proof}
    It remains to establish the invariance under the Fourier transform, which follows from the alternative definition of $I(s)$ as the alternating sum
    \begin{eqnarray*}
        \int_{G(F)\backslash G(\mathbb{A})} \sum_P(-1)^{\ell(P)} K_P(g;f) \sum_{v^*\in(V^{U_P})'(F)} \Phi(v^*g)\big|\det(g)\big|^s \mathrm{d}g,
    \end{eqnarray*}
    together with the observation that $K_P(g;f)=K_P(g;\widehat{f})$.
\end{proof}

\begin{remark}
    Let $\mathfrak{o}\subset\mathfrak{g}(F)$ be a conjugacy class with characteristic polynomial $p(t)=\prod_iq_i(t)^{m_i}$ and invariant factors $p_r\mid\dots\mid p_1$ where $p_j(t)=\prod_iq_i(t)^{m_{ij}}$ and $q_i(t)\in F[t]$ are all distinct and irreducible, then each irreducible factor $q_i$ determines a partition $\lambda_i=[\dots,m_{ij},\dots,m_{i1}]\vdash m_i$. In particular
    \begin{itemize}
        \item $\mathfrak{o}=\mathfrak{o}_\mathrm{reg}$ is regular if and only if all $\lambda_i=[m_i]$ with Young diagram
        
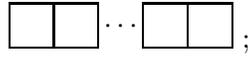
\begin{eqnarray*}
            \begin{ytableau}
            ~&~&\none[\cdots]&~&~
            \end{ytableau}~;
        \end{eqnarray*}
        \item $\mathfrak{o}=\mathfrak{o}_\mathrm{ss}$ is semisimple if and only if all $\lambda_i=[1,\dots,1]$ with Young diagram
        
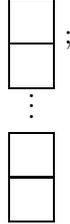
\begin{eqnarray*}
            \begin{ytableau}
            ~\\~\\\none[\vdots]\\~\\~
            \end{ytableau}~;
        \end{eqnarray*}
        \item $\mathfrak{o}=\mathfrak{o}_\mathrm{rs}$ is regular semisimple if and only if all $m_i=1$ and $\lambda_i=[1]$ with Young diagram
        
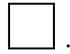
\begin{eqnarray*}
            \begin{ytableau}
            ~
            \end{ytableau}~.
        \end{eqnarray*}
        \end{itemize}
        In general, the regularized mirabolic integral satisfies the functional equation
        \begin{eqnarray*}
        I(s)
        &=&
        I(1-s)
        \end{eqnarray*}
        under the transformations $f\leftrightarrow f^\intercal$ and $\Phi\leftrightarrow\widehat{\Phi}$, which interchanges each partition $\lambda_i\vdash m_i$ with the conjugate partition $\overline{\lambda_i}\vdash m_i$ which is obtained by reflecting the Young diagram along the main diagonal.

        More precisely, Proposition \ref{gl(n) singular orbital integral proposition} could be reformulated in terms of Young diagrams which states that the regularized mirabolic integral $I_{\mathfrak{o}}(s)$ is an entire multiple of
        \begin{eqnarray*}
        \prod_i\prod_{j,k} \zeta_{E_i}\big(h_{\lambda_i}(j,k)s-a_{\lambda_i}(j,k)\big),
        \end{eqnarray*}
        where $E_i\simeq F[t]/(q_i)$ and $h_\lambda(j,k),a_\lambda(j,k)$ denote the hook length and arm length of the $(j,k)$th box of the Young diagram of a partition $\lambda\vdash m$
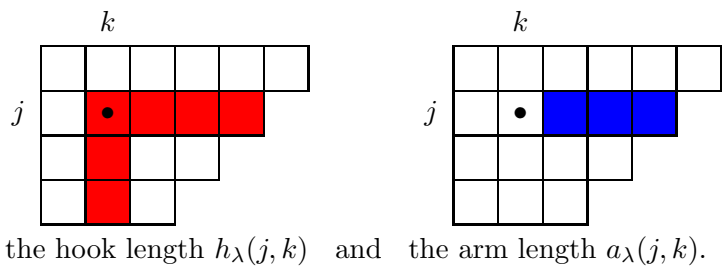
        \begin{eqnarray*}
            \begin{ytableau}
            \none&\none&\none[k]\\
            \none&~&~&~&~&~&~\\
            \none[j]&~&*(red)\bullet&*(red)~&*(red)~&*(red)~\\
            \none&~&*(red)~&~&~\\
            \none&~&*(red)~&~
            \end{ytableau}
            &&
            \begin{ytableau}
            \none&\none&\none[k]\\
            \none&~&~&~&~&~&~\\
            \none[j]&~&\bullet&*(blue)~&*(blue)~&*(blue)~\\
            \none&~&~&~&~\\
            \none&~&~&~
            \end{ytableau}\\
            \textrm{the~hook~length~}h_\lambda(j,k)
            &\textrm{and}&
            \textrm{the~arm~length~}a_\lambda(j,k).
        \end{eqnarray*}
\end{remark}

\appendix
\section{Appendix: The local theory}
\label{appendix}

\paragraph{Local preliminaries}
For this appendix let $F_v$ be a non-archimedean local field and let $X_v$ denote the totally disconnected locally compact space $X(F_v)$ consisting of the $F_v$-valued points of an algebraic variety $X$.

Let $f\in\mathcal{S}(\mathfrak{g}_v)$, then by the Weyl integration formula
\begin{eqnarray*}
    \int_{\mathfrak{g}_v} f(X)\mathrm{d}X
    &=&
    \sum_{\lambda\vdash n} \sum_{E_v^\times\subset M_{\lambda,v}} \frac{1}{|W(M_{\lambda,v},E_v^\times)|} \int_{E_{\mathrm{reg},v}} \int_{E_v^\times\backslash G_v} f\big(X\cdot\mathrm{ad}(g)\big) \mathrm{d}g \big|\Delta(X)\big|_v\mathrm{d}X
\end{eqnarray*}
where the summation $\sum_{E_v^\times}$ ranges over all conjugacy classes of maximal tori $E_v^\times\subset M_{\lambda,v}$ elliptic modulo the center of $M_{\lambda,v}$, $E_v=\mathfrak{e}_v^\times\subset\mathfrak{g}_v$ denotes the Lie algebra of $E_v^\times$ and $E_{\mathrm{reg},v}\subset E_v$ denotes the regular locus defined by $\Delta(X)\neq0$ where $\Delta(X)$ denotes the discriminant of the characteristic polynomial of $X$, and $W(M_{\lambda,v},E_v^\times)$ denotes the relative Weyl group of $E_v^\times$ with respect to $M_{\lambda,v}$.

Fix a non-trivial additive character $\psi_v:F_v\to\mathbb{C}^\times$. Let $\langle\cdot,\cdot\rangle_v$ be a non-degenerate bilinear form on a vector space $W_v$ equipped with the self-dual Haar measure $\mathrm{d}w$. For all $f\in\mathcal{S}(W_v)$, define its Fourier transform by
\begin{eqnarray*}
    \widehat{f}(v)
    &=&
    \int_{W_v} f(w)\overline{\psi_v}(\langle v,w\rangle_v)\mathrm{d}w.
\end{eqnarray*}
Then the Fourier transform satisfies the inversion formula
\begin{eqnarray*}
    \widehat{\widehat{f_1}}(w)
    &=&
    f_1(-w)
\end{eqnarray*}
and the Plancherel formula
\begin{eqnarray*}
    \int_{W_v} \widehat{f_1}(w)f_2(w) \mathrm{d}w
    &=&
    \int_{W_v} f_1(w)\widehat{f_2}(w) \mathrm{d}w
\end{eqnarray*}
for all $f_1,f_2\in\mathcal{S}(W_v)$.

\paragraph{The local trace formula of Waldspurger}
Let $f_1,f_2\in\mathcal{S}(\mathfrak{g}_v)$, define the local kernel function
\begin{eqnarray*}
    K_v(g;f_1,f_2)
    &=&
    \int_{\mathfrak{g}_v} f_1(X)f_2\big(X\cdot\mathrm{ad}(g)\big) \mathrm{d}X
\end{eqnarray*}
for $g\in G_v$, then
\begin{eqnarray*}
    K_v(g;\widehat{f_1},f_2)
    &=&
    K_v(g;f_1,\widehat{f_2})
\end{eqnarray*}
by the Plancherel formula.

Applying the results of Arthur \cite{A91}, Waldspurger \cite{W95} has determined the asymptotic behavior of the divergent integral
\begin{eqnarray*}
    \int_{Z_v\backslash G_v} K_v(g;f_1,f_2) \mathrm{d}g
\end{eqnarray*}
over an increasing system of compact subsets of $Z_v\backslash G_v$. The resultant local trace formula is an identity of the form
\begin{eqnarray*}
    J(\widehat{f_1},f_2)
    &=&
    J(f_1,\widehat{f_2})
\end{eqnarray*}
between absolutely convergent integrals defined by
\begin{eqnarray*}
    J(f_1,f_2)
    &=&
    \sum_{\lambda\vdash n} \sum_{E_v^\times\subset M_{\lambda,v}} \frac{(-1)^{\ell(P_\lambda)}}{|W(M_{\lambda,v},E_v^\times)|} \int_{E_{\mathrm{reg},v}} J_{M_\lambda,v}(X;f_1,f_2) \mathrm{d}X
\end{eqnarray*}
where the distributions $J_{M_\lambda,v}(X;f_1,f_2)$ are combinations of weighted orbital integrals $J_{M_\lambda,v}^{M_{P_1}}\big(X;(f_1)_{P_1,v}\big)$ and $J_{M_\lambda,v}^{M_{P_2}}\big(X;(f_2)_{P_2,v}\big)$ of the parabolic descent of $f_1$ and $f_2$ along all standard parabolic subgroups $M_\lambda\subset P_1,P_2\subset G$.
\begin{itemize}
    \item If $\lambda=[n]$ and $E_v^\times\subset G_v$ is elliptic modulo $Z_v$ where $E_v/F_v$ is a field extension of degree $n$, then
    \begin{eqnarray*}
        J_{G,v}(X;f_1,f_2)
        &=&
        |\Delta(X)|_v \int_{E_v^\times\backslash G_v} f_1\big(X\cdot\mathrm{ad}(g_1)\big) \mathrm{d}g_1 \int_{E_v^\times\backslash G_v} f_2\big(X\cdot\mathrm{ad}(g_2)\big) \mathrm{d}g_2\\
        &=&
        I_{G,v}^G(X;f_1)I_{G,v}^G(X;f_2)
    \end{eqnarray*}
    where
    \begin{eqnarray*}
        I_{G,v}^G(X;f)
        &=&
        |\Delta(X)|_v^{1/2} \int_{E_v^\times\backslash G_v} f\big(X\cdot\mathrm{ad}(g)\big) \mathrm{d}g
    \end{eqnarray*}
    denotes a normalized elliptic orbital integral.
    \item Waldspurger \cite{W95} has also introduced another version of the local trace formula
    \begin{eqnarray*}
        &&
        \sum_\lambda \sum_{E_v^\times} \frac{(-1)^{\ell(P_\lambda)}}{|W(M_{\lambda,v},E_v^\times)|} \int_{E_{\mathrm{reg},v}} I_{M_\lambda,v}^G(X;\widehat{f_1})I_{G,v}^G(X;f_2) \mathrm{d}X\\
        &=&
        \sum_\lambda \sum_{E_v^\times} \frac{(-1)^{\ell(P_\lambda)}}{|W(M_{\lambda,v},E_v^\times)|} \int_{E_{\mathrm{reg},v}} I_{G,v}^G(X;f_1)I_{M_\lambda,v}^G(X;\widehat{f_2}) \mathrm{d}X
    \end{eqnarray*}
    where $I_{M,v}^G(X;f)$ are conjugation invariant analogues of the weighted orbital integrals $J_{M,v}^G(X;f)$.
\end{itemize}

\paragraph{Local Tate integrals}
Let $\Phi\in\mathcal{S}(F_v)$, then the local Tate integral $I_v(s;\Phi)$ is a local integral of the form
\begin{eqnarray*}
    I_v(s;\Phi)
    &=&
    \int_{F_v^\times} \Phi(x) |x|_v^s\mathrm{d}^\times x
\end{eqnarray*}
for all $s\in\mathbb{C}$ such that the integral converges absolutely. In addition to the analytic properties of the global zeta functions, Tate \cite{T50} has also established the following local analytic properties of $I_v(s;\Phi)$:
\begin{itemize}
    \item the integral $I_v(s;\Phi)$ converges absolutely on the half-plane $\mathrm{Re}(s)>0$ and continuous to a meromorphic function on the entire complex plane, which will also be denoted by $I_v(s,\Phi)$;
    \item there exists a Laurent polynomial $p(t)\in\mathbb{C}[t,t^{-1}]$ such that
    \begin{eqnarray*}
        I_v(s;\Phi)
        &=&
        \frac{p(q_v^{-s})}{1-q_v^{-s}}
    \end{eqnarray*}
    where $q_v=|F_v/O_v|$ denotes the cardinality of the residue field of $F_v$;
    \item the meromorphic functions $I_v(s;\Phi_1)$ and $I_v(s;\Phi_2)$ satisfy the local functional equation
    \begin{eqnarray*}
        I_v(s;\Phi_1)I_v(1-s;\widehat{\Phi}_2)
        &=&
        I_v(1-s;\widehat{\Phi}_1)I_v(s;\Phi_2)
    \end{eqnarray*}
    for all $\Phi_1,\Phi_2\in\mathcal{S}(F_v)$, which follows from the Plancherel formula.
\end{itemize}
Parallel to the global Tate integrals, Tate \cite{T50} has also introduced local Tate integrals twisted by a local multiplicative character $\chi_v:F_v^\times\to\mathbb{C}^\times$, which will also be excluded from this paper for simplicity.

\begin{definition}
    Let $\Phi_1,\Phi_2\in\mathcal{S}(V_v)$, then the local mirabolic Eisenstein integral $E_v(g,s;\Phi_1,\Phi_2)$ is defined by
    \begin{eqnarray*}
        E_v(g,s;\Phi_1,\Phi_2)
        &=&
        |\det(g)|_v^s \int_{F_v^\times} \int_{V_v} \Phi_1(v^*) \Phi_2(zv^*g) \mathrm{d}v^*|z|_v^{ns}\mathrm{d}^\times z
    \end{eqnarray*}
    which converges absolutely for $\mathrm{Re}(s)>0$ and continues meromorphically to the entire complex plane, where $g\in Z_v\backslash G_v$.
\end{definition}

\begin{proposition}
     Let $f_1,f_2\in\mathcal{S}(\mathfrak{g}_v)$ and $\Phi_1,\Phi_2\in\mathcal{S}(V_v)$, then the integral
     \begin{eqnarray}
         \label{gl(n) local integral}
         &&
         \int_{Z_v\backslash G_v} K_v(g;f_1,f_2) E(g,s;\Phi_1,\Phi_2) \mathrm{d}g\\
         &=&
         \sum_{\lambda\vdash n} \sum_{E_v^\times\subset M_{\lambda,v}} \frac{1}{|W(M_{\lambda,v},E_v^\times)|} \int_{E_{\mathrm{reg},v}} I_v(X,s;f_1,f_2,\Phi_1,\Phi_2) \mathrm{d}X \nonumber
     \end{eqnarray}
     where
     \begin{eqnarray*}
         I_v(X,s;f_1,f_2,\Phi_1,\Phi_2)
         &=&
         \int_{E_v^\times\backslash(G_v\times G_v)} f_1\big(X\cdot\mathrm{ad}(g_1)\big)f_2\big(X\cdot\mathrm{ad}(g_2)\big)\times\\
         &&\quad\times
         \sideset{}{_{V_v}'}\int \Phi_1(v^*g_1)\Phi_2(v^*g_2) \mathrm{d}v^*\times\\
         &&\qquad\times
         \big|\det(g_1)\big|_v^{1-s}\big|\det(g_2)\big|_v^s\big|\Delta(X)\big|_v \mathrm{d}g_1\mathrm{d}g_2
     \end{eqnarray*}
     where the primed integral is over all $v^*\in V_v$ such that $(X,v^*)\in(\mathfrak{g}\times V')_{\mathrm{reg},v}$, converges absolutely if $0<\mathrm{Re}(s)<1$ and continuous meromorphically over the entire complex plane.
\end{proposition}

\begin{proof}
    Restricting to the subset $(\mathfrak{g}\times V')_{\mathrm{reg},v}\subset\mathfrak{g}_v\times V_v$ of full measure, choosing $e_{E_v}^*\in V_v$ for each $E_v^\times$ such that the orbit $e_{E_v}^*E_v^\times\subset V_v$ is Zariski open, and applying the Weyl integration formula, then the integral (\ref{gl(n) local integral}) is equal to
    \begin{eqnarray*}
         &=&
         \int_{G_v} \int_{\mathfrak{g}_v}f_1(X)f_2\big(X\cdot\mathrm{ad}(g)\big)\mathrm{d}X \int_{V_v}\Phi_1(v^*)\Phi_2(v^*g)\mathrm{d}v^* \big|\det(g)\big|_v^s\mathrm{d}g\\
         &=&
         \sum_{\lambda\vdash n} \sum_{E_v^\times\subset M_{\lambda,v}} \frac{1}{|W(M_{\lambda,v},E_v^\times)|} \int_{G_v} \int_{E_v^\times\backslash G_v} \Big( \int_{E_{\mathrm{reg},v}} \\
         &&\quad\times f_1\big(X\cdot\mathrm{ad}(g_1)\big) f_2\big(X\cdot\mathrm{ad}(g_1g)\big) \sideset{}{_{V_v}'}\int \Phi_1(v^*g_1) \Phi_2(v^*g_1g)\mathrm{d}v^* \times\\
         &&\qquad\times \big|\Delta(X)\big|_v\mathrm{d}X\Big) \mathrm{d}g_1 \big|\det(g)\big|_v^s\mathrm{d}g\\
         &=&
         \sum_{\lambda\vdash n} \sum_{E_v^\times\subset M_{\lambda,v}} \frac{1}{|W(M_{\lambda,v},E_v^\times)|} \int_{G_v} \int_{G_v} \\
         &&\quad\times \Big( \int_{E_{\mathrm{reg},v}} f_1\big(X\cdot\mathrm{ad}(g_1)\big) f_2\big(X\cdot\mathrm{ad}(g_1g)\big) \big|\Delta(X)\big|_v\mathrm{d}X \times\\
         &&\qquad\times \Phi_1(e_{E_v}^*g_1) \Phi_2(e_{E_v}^*g_1g) \Big) \big|\det(g_1)\big|_v\mathrm{d}g_1 \big|\det(g)\big|_v^s\mathrm{d}g\\
         &=&
         \sum_{\lambda\vdash n} \sum_{E_v^\times\subset M_{\lambda,v}} \frac{1}{|W(M_{\lambda,v},E_v^\times)|} \int_{G_v} \int_{G_v} \\
         &&\quad\times \Big( \int_{E_{\mathrm{reg},v}} f_1\big(X\cdot\mathrm{ad}(g_1)\big) f_2\big(X\cdot\mathrm{ad}(g_2)\big) \big|\Delta(X)\big|_v\mathrm{d}X \times\\
         &&\qquad\times \Phi_1(e_{E_v}^*g_1) \Phi_2(e_{E_v}^*g_2) \Big) \big|\det(g_1)\big|_v^{1-s} \big|\det(g_2)\big|_v^s \mathrm{d}g_1\mathrm{d}g_2.
     \end{eqnarray*}
     If $\lambda=[n_r,\dots,n_1]$ and $E_v^\times$ elliptic in $M_{\lambda,v}$, then 
     \begin{eqnarray*}
         E_v^\times\simeq 
         \left[\begin{matrix}
         E_{r,v}^\times&\cdots&0\\
         \vdots&\ddots&\vdots\\
         0&\cdots&E_{1,v}^\times
         \end{matrix}\right]
         &\subset&
         \left[\begin{matrix}
         \mathrm{GL}_{n_r,v}&\cdots&0\\
         \vdots&\ddots&\vdots\\
         0&\cdots&\mathrm{GL}_{n_1,v}
         \end{matrix}\right]
         \simeq M_{\lambda,v}
     \end{eqnarray*}
     where $E_{i,v}/F_v$ is a field extension of degree $n_i$ for $i=1,\dots,r$, and
     \begin{eqnarray*}
         I_v(g,s;\Phi)
         &=&
         \big|\det(g)\big|_v^s \int_{E_{r,v}^\times}\!\!\!...\!\!\!\int_{E_{1,v}^\times} \Phi\big( [e_{E_r}^*x_r,\dots,e_{E_1}^*x_1] g\big) \prod_{i=1}^r\big|\det(x_i)\big|_v^s\mathrm{d}x_i
     \end{eqnarray*}
     is a family of products of local Tate integrals which is locally constant and compactly supported in $g\in E_v^\times\backslash G_v$. Hence the integral (\ref{gl(n) local integral}) is equal to
     \begin{eqnarray}
         \label{gl(n) local tate integral}
         &=&
         \sum_{\lambda\vdash n} \sum_{E_v^\times\subset M_{\lambda,v}} \frac{1}{|W(M_{\lambda,v},E_v^\times)|} \int_{E_{\mathrm{reg},v}} \Big(\int_{E_v^\times\backslash G_v} \int_{E_v^\times\backslash G_v} \\
         &&\quad\times  f_1\big(X\cdot\mathrm{ad}(g_1)\big) f_2\big(X\cdot\mathrm{ad}(g_2)\big) I_v(g_1,1-s,\Phi_1) I_v(g_2,s,\Phi_2) \mathrm{d}g_1\mathrm{d}g_2\Big) \big|\Delta(X)\big|_v\mathrm{d}X\nonumber
     \end{eqnarray}
     which converges absolutely on the strip $0<\mathrm{Re}(s)<1$ and continuous meromorphically over the entire complex plane, where
     \begin{eqnarray*}
         I_v(g_1,1-s,\Phi_1) I_v(g_2,s,\Phi_2)
         &=&
         \int_{E_v^\times}\sideset{}{_{V_v}'}\int \Phi_1(v^*x^{-1}g_1)\Phi_2(v^*xg_2) \mathrm{d}v^*\mathrm{d}x
     \end{eqnarray*}
     and the integrand in (\ref{gl(n) local tate integral}) is equal to $I_v(X,s;f_1,f_2,\Phi_1,\Phi_2)$.
\end{proof}

\begin{corollary}
     If $E_v^\times$ is an elliptic torus in $M_{\lambda,v}$ where $\lambda=[n_r,\dots,n_1]$ and $X\in E_{\mathrm{reg},v}$, then $I_v(X,s;f_1,f_2,\Phi_1,\Phi_2)$ is the product of
     \begin{eqnarray*}
         \zeta_{E_v}(q_v^{-s})
         &=&
         \prod_{i=1}^r \frac{1}{(1-q_v^{-n_i+n_is})(1-q_v^{-n_is})}
     \end{eqnarray*}
     by an entire Laurent polynomial of the form $p(q_v^{-s})\in\mathbb{C}[q_v^{-s},q_v^s]$.
\end{corollary}

\begin{proof}
    By the analytic properties of local Tate integrals,
    \begin{eqnarray*}
        I_v(X,s;f_1,f_2,\Phi_1,\Phi_2)
        &=&
        \int_{E_v^\times\backslash G_v} \int_{E_v^\times\backslash G_v} f_1\big(X\cdot\mathrm{ad}(g_1)\big) f_2\big(X\cdot\mathrm{ad}(g_2)\big) \big|\Delta(X)\big|_v\times\\
        &&\qquad\times
        I_v(g_1,1-s,\Phi_1) I_v(g_2,s,\Phi_2) \mathrm{d}g_1\mathrm{d}g_2
    \end{eqnarray*}
    could be interpreted as the integral of a vector-valued function in
    \begin{eqnarray*}
        C_c^\infty\Big(E_v^\times\backslash G_v\times E_v^\times\backslash G_v~,~\zeta_{E_v}(q_v^{-s})\cdot\mathbb{C}[q^{-s},q^s]\Big)
    \end{eqnarray*}
    where $\zeta_{E_v}(q_v^{-s})\cdot\mathbb{C}[q_v^{-s},q_v^s]\subset\mathbb{C}(q_v^{-s})$ denotes the $\mathbb{C}[q_v^{-s},q_v^s]$-fractional ideal generated by $\zeta_{E_v}(q_v^{-s})$. Since the integral only has finitely many summands, the corollary follows.
\end{proof}

\begin{theorem}[Local mirabolic trace formula]
    Let $\Phi_1,\Phi_2\in\mathcal{S}(V_v)$, then the distribution $I_v(s)\in\mathcal{S}'(\mathfrak{g}_v\times\mathfrak{g}_v)$ defined by
    \begin{eqnarray*}
        I_v(s;f_1,f_2)
        &=&
        \sum_{\lambda\vdash n} \sum_{E_v^\times\subset M_{\lambda,v}} \frac{1}{|W(M_{\lambda,v},E_v^\times)|} \int_{E_{\mathrm{reg},v}} I_v(X,s;f_1,f_2,\Phi_1,\Phi_2) \mathrm{d}X
    \end{eqnarray*}
    where the summation $\sum_{E_v^\times}$ ranges over all conjugacy classes of maximal tori $E_v^\times\subset M_{\lambda,v}$ elliptic modulo the center of $M_{\lambda,v}$
    \begin{itemize}
        \item converges absolutely if $0<\mathrm{Re}(s)<1$ and continues meromorphically to a sum of entire multiples of products of local zeta functions, and
        \item satisfies the identity
        \begin{eqnarray*}
            I_v(s;\widehat{f_1},f_2)
            &=&
            I_v(s;f_1,\widehat{f_2}).
        \end{eqnarray*}
    \end{itemize}
\end{theorem}

\begin{proof}
    It remains to establish the identity under the Fourier transforms, which follows from the alternative definition
    \begin{eqnarray*}
        I_v(s;f_1,f_2)
        &=&
        \int_{Z_v\backslash G_v} K_v(g;f_1,f_2) E(g,s;\Phi_1,\Phi_2) \mathrm{d}g
    \end{eqnarray*}
    together with the identity $K_v(g;\widehat{f_1},f_2)=K_v(g;f_1,\widehat{f_2})$.
\end{proof}

\end{document}